\documentclass[a4paper,12pt,reqno]{amsart}

\usepackage[utf8]{inputenc}  \usepackage[T1]{fontenc}   
\usepackage[english]{babel}
\usepackage{amsmath}
\usepackage{amsthm}
\usepackage{amsfonts} 
\usepackage{amssymb}
\usepackage{mathtext}
\usepackage{graphicx}
\usepackage{hyperref} 
\usepackage{array}
\usepackage{adjustbox}
\usepackage{booktabs}
\usepackage{multirow}
\usepackage{caption}
\usepackage{subcaption}

\newcommand{\G} {\mathcal{G}}
\newcommand{\T} {\mathcal{T}}
\newcommand{\I} {\mathcal{I}}
\newcommand{\J} {\mathcal{J}}
\newcommand{\D} {\mathcal{D}}

\newcommand{\F} {\mathcal{F}}
\newcommand{\Z} {\mathbb{Z}}
\newcommand{\N} {\mathbb{N}}

\newcommand{\Sig} {\mathfrak{S}}

\theoremstyle{plain}
\newtheorem{theorem}{Theorem}[section]
\newtheorem{lemma}[theorem]{Lemma}
\newtheorem{corollary}[theorem]{Corollary}
\newtheorem{proposition}[theorem]{Proposition}
\newtheorem{fact}{Fact}[section]

\theoremstyle{definition}
\newtheorem{definition}{Definition}[section]
\newtheorem{example}{Example}[section]

\theoremstyle{remark}
\newtheorem{remark}{Remark}[section]

\renewenvironment{proof}{\noindent{\bf Proof.}}{\qed}

\author{Ange Bigeni}
\address{Institut Camille Jordan, Universit\'e Claude Benard Lyon 1 (France)
}
\email{bigeni@math.univ-lyon1.fr}

\title[Combinatorial study of the Dellac configurations]{Combinatorial study of the Dellac configurations and the $q$-extended normalized median Genocchi numbers}

\date{}

\begin{document}

\maketitle


\begin{abstract}
In two recent papers (\textit{Mathematical Research Letters,18(6):1163--1178,2011} and \textit{European J. Combin.,33(8):1913--1918,2012}), Feigin proved that the Poincaré polynomials of the degenerate flag varieties have a combinatorial interpretation through the Dellac configurations, and related them to the $q$-extended normalized median Genocchi numbers $\bar{c}_n(q)$ introduced by Han and Zeng, mainly by geometric considerations.
In this paper, we give combinatorial proofs of these results by constructing statistic-preserving bijections between the Dellac configurations and two other combinatorial models of $\bar{c}_n(q)$.

\end{abstract}


\section{Introduction}

The Genocchi numbers $(G_{2n})_{n \geq 1} = (\textcolor{red}{1},\textcolor{red}{1},\textcolor{red}{3},\textcolor{red}{17},\textcolor{red}{155},\hdots)$ and the median Genocchi numbers $(H_{2n+1})_{n \geq 0} = (\textcolor{blue}{1},\textcolor{blue}{2},\textcolor{blue}{8},\textcolor{blue}{56},\textcolor{blue}{608},\hdots)$ are the entries $g_{2n-1,n}$ and $g_{2n+2,1}$ respectively in the Seider triangle $(g_{i,j})_{0 \leq j \leq i}$ (see Figure \ref{seidertriangle}) defined by 
\begin{align*}
g_{2p-1,j} &= g_{2p-1,j-1} + g_{2p-2,j},\\
g_{2p,j} &= g_{2p-1,j} + g_{2p,j+1},
\end{align*}
with $g_{1,1} = 1$ and $g_{i,j} = 0$ whenever $i < j$ or $j=0$ (see \cite{DV}).

\begin{figure}[!h] \center
\begin{tabular}{c|cccccccccccccccccccccc}
$\vdots$ & & & & & & & & & & & & & & & & & \\
5 & & & & & & & & & & & & & & & & & \textcolor{red}{155} & $\rightarrow$ & 155 & $\rightarrow$ & \reflectbox{$\ddots$} \\ 
 & & & & & & & & & & & & & & & & & $\uparrow$ & & $\downarrow$ \\
4 & & & & & & & & & & & & & \textcolor{red}{17} & $\rightarrow$ & 17 & $\rightarrow$ & 155 & $\rightarrow$ & 310 & $\rightarrow$ & $\hdots$ \\ 
 & & & & & & & & & & & & & $\uparrow$ & & $\downarrow$ & & $\uparrow$ & & $\downarrow$ \\
3 & & & & & & & & & \textcolor{red}{3} & $\rightarrow$ & 3 & $\rightarrow$ & 17 & $\rightarrow$ & 34 & $\rightarrow$ & 138 & $\rightarrow$ & 448 & $\rightarrow$ &  $\hdots$ \\ 
  & & & & & & & & & $\uparrow$ & & $\downarrow$ & & $\uparrow$ & & $\downarrow$ & & $\uparrow$ & & $\downarrow$ \\
2 & & & & & \textcolor{red}{1} & $\rightarrow$ & 1 & $\rightarrow$ & 3 & $\rightarrow$ & 6 & $\rightarrow$ & 14 & $\rightarrow$ & 48 & $\rightarrow$ & 104 & $\rightarrow$ & 552 & $\rightarrow$ &  $\hdots$ \\ 
 & & & & & $\uparrow$ & & $\downarrow$ & & $\uparrow$ & & $\downarrow$ & & $\uparrow$ & & $\downarrow$ & & $\uparrow$ & & $\downarrow$ \\
1  & \textcolor{red}{1} & $\rightarrow$ & \textcolor{blue}{1} & $\rightarrow$ & 1 & $\rightarrow$ & \textcolor{blue}{2} & $\rightarrow$ & 2 & $\rightarrow$ & \textcolor{blue}{8} & $\rightarrow$ & 8 & $\rightarrow$ & \textcolor{blue}{56} & $\rightarrow$ & 56 & $\rightarrow$ & \textcolor{blue}{608} & $\rightarrow$ & $\hdots$ \\ 
\hline
$j/i$ & 1 & & 2 & & 3 & & 4 & & 5 & & 6 & & 7 & & 8 & & 9 & & 10& & $\hdots$
\end{tabular}
\caption{Seider generation of the Genocchi numbers.}
\label{seidertriangle}
\end{figure}
\hspace*{-5.9mm} 
It is well known that $H_{2n+1}$ is divisible by $2^n$ (see \cite{Barsky}) for all $n \geq 0$. The \textit{normalized median Genocchi numbers} $(h_n)_{n \geq 0} = (1,1,2,7,38,\hdots)$ are the positive integers defined by 
$$h_n = H_{2n+1}/2^n.$$
Dumont \cite{Dumont2} gave several combinatorial models of the Genocchi numbers and the median Genocchi numbers, among which are the \textit{Dumont permutations}.
We denote by $\Sig_n$ the set of permutations of the set $[n] := \{1,2,\hdots,n\}$, and by inv$(\sigma)$ the number of inversions of a permutation $\sigma \in \Sig_n$, \textit{i.e.}, the quantity of pairs $(i,j) \in  [n]^2$ with $i < j$ and $\sigma(i)>\sigma(j)$. Broadly speaking, the number of inversions inv$(w)$ of a word $w = l_1 l_2 \hdots l_n$ with $n$ letters in the alphabet $\N$ is the quantity of pairs $(i,j) \in [n]^2$ such that $i < j$ and $l_i > l_j$. In particular, the number inv$(\sigma)$ associated with a permutation $\sigma \in \Sig_n$ is the quantity inv$(w)$ associated with the word $w = \sigma(1) \sigma(2) \hdots \sigma(n)$.

\begin{definition} A \text{Dumont permutation} of order $2n$ is a permutation $\sigma \in~\Sig_{2n}$ such that $\sigma(2i) < 2i $ and $\sigma(2i-1) > 2i-1$ for all $i$. We denote by $\D_n$ the set of these permutations.
\end{definition}

It is well-known (see \cite{Dumont2}) that $H_{2n+1} = |\D_{n+1}|$ for all $n \geq 0$.
In \cite{HZ2}, Han and Zeng introduced the set $\G_n''$ of \textit{normalized Genocchi permutations}, which consist of permutations $\sigma \in \D_n$ such that for all $j \in [n-1]$, the two integers $\sigma^{-1}(2j)$ and $\sigma^{-1}(2j+1)$ have the same parity if and only if $\sigma^{-1}(2j) < \sigma^{-1}(2j+1)$, and they proved that $h_{n} = |\G_{n+1}''|$ for all $n \geq 0$.
The number $h_n$ also counts the Dellac configurations of size $n$ (see \cite{Feigin2}).

\begin{definition} \label{definitiondellacconfiguration} A Dellac configuration of size $n$ is a tableau of width $n$ and height $2n$ which contains $2n$ dots between the lines $y=x$ and $y=n+x$, such that each row contains exactly one dot and each column contains exactly two dots. Let $DC(n)$ be the set of Dellac configurations of size $n$. An \textit{inversion} of $C \in DC(n)$ is a pair $(d_1,d_2)$ of dots whose Cartesian coordinates in $C$ are respectively $(j_1,i_1)$ and $(j_2,i_2)$ such that $j_1 < j_2$ and $i_1 > i_2$. We denote by $\text{inv}(C)$ the number of inversions of $C$. For example, the tableau depicted in Figure \ref{exempledellac} is a Dellac configuration $C \in DC(3)$ with $\text{\text{inv}}(C) = 2$ inversions (represented by two segments).

\begin{figure}[!h] \centering \label{exempledellac}
\includegraphics[width=1.5cm]{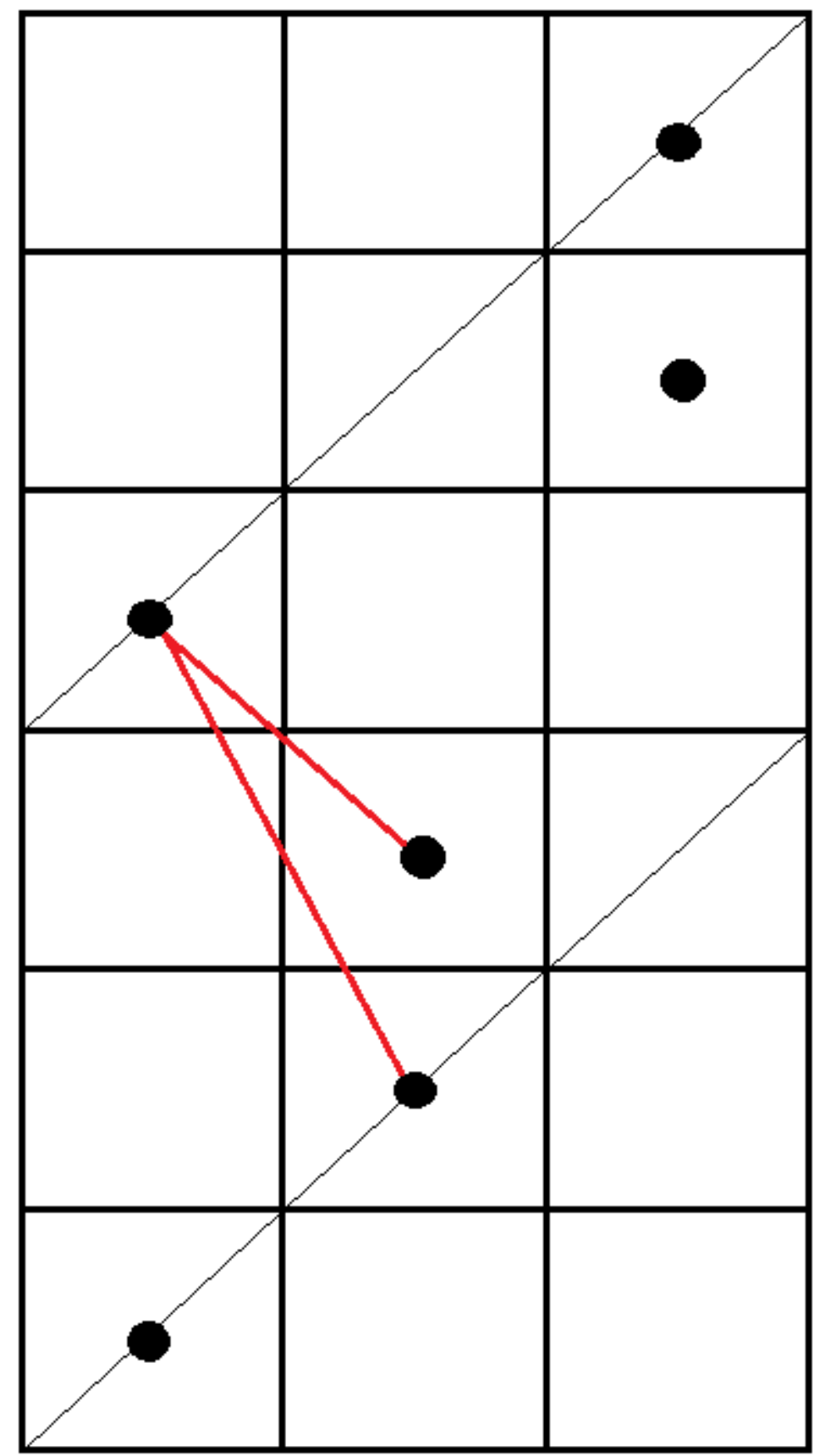}
\caption{Dellac configuration $C \in DC(3)$ with inv$(C) = 2$ inversions.}
\label{exempledellac}
\end{figure}
\end{definition}

In \cite{HZ2,HZ1}, Han and Zeng defined the \textit{$q$-Gandhi polynomials of the second kind} $(C_n(x,q))_{n \geq 1}$ by
$C_1(x,q) = 1$ and $C_{n+1}(x,q) = (1+qx) \Delta_q(xC_n(x,q))$, where 
$$\Delta_q P(x)=(P(1+qx) - P(x))/(1+qx -x)$$
for all polynomial $P(x)$.
They proved that the polynomials $C_n(1,q)$ are $q$-analogs of the median Genocchi numbers $(C_n(1,1) = H_{2n-1})$. Furthermore, they gave a combinatorial interpretation of $C_n(1,q)$ through $\D_n$.

\begin{theorem}[Han and Zeng, 1997] \label{theointerpretationcombinatoireCn1q}
Let $n \geq 1$. For all $\sigma \in \D_n$, we define $st(\sigma)$ as the quantity
\begin{equation} \label{definitionst}
st(\sigma) = n^2 -  \sum_{i=1}^n \sigma(2i) - inv(\sigma^o) - inv(\sigma^e)
\end{equation}
where $\sigma^o$ and $\sigma^e$ are the two words $\sigma(1) \sigma(3) \hdots \sigma(2n-~1)$ and $\sigma(2) \sigma(4) \hdots \sigma(2n)$ respectively. Then, the polynomial $C_n(1,q)$ has the following combinatorial interpretation:
\begin{equation} \label{formuleCn}
C_n(1,q) = \sum_{\sigma \in \D_n} q^{st(\sigma)}.
\end{equation}
\end{theorem}

By introducing the subset $\G_n'' \subset \D_n$ of normalized Genocchi permutations and using the combinatorial interpretation provided by Theorem \ref{theointerpretationcombinatoireCn1q}, Han and Zeng proved combinatorially that the polynomial $(1+q)^{n-1}$ divides $C_n(1,q)$, which gives birth to polynomials $(\bar{c}_n(q))_{n \geq 1}$ defined by
\begin{equation} \label{definitioncnq}
\bar{c}_n(q) = C_n(1,q)/(1+q)^{n-1}.
\end{equation} 
This divisibility had previously been proved in the same paper with a continued fraction approach, as a corollary of the following theorem and a well-known result on continued fractions (see \cite{Flajolet}).

\begin{theorem}[Han and Zeng, 1997] \label{theogeneratingfunctioncn}
The generating function of the sequence $(\bar{c}_{n+1}(q))_{n \geq 0}$ is 
\begin{equation} \label{expansioncn+1}
\sum_{n \geq 0} \bar{c}_{n+1}(q) t^n = \dfrac{1}{1- \dfrac{\lambda_1 t}{1- \dfrac{\lambda_2 t}{1-\dfrac{\lambda_3 t}{\ddots}}}}
\end{equation}
where the sequence $(\lambda_n)_{n \geq 1}$ is defined by $\lambda_{2p-1} = (1-q^{p+1})(1-q^p)/(1-q^2)(1-q)$ and $\lambda_{2p} =~q \lambda_{2p-1}$ for all $p \geq 1$.
\end{theorem}

The polynomials $(\bar{c}_n(q))_{n \geq 1}$ are $q$-refinements of normalized median Genocchi numbers: $\bar{c}_n(1) = h_{n-1}$ for all $n \geq 1$. They are named \textit{$q$-extended normalized median Genocchi numbers}. In \S \ref{sec:preliminaries}, we give a combinatorial interpretation of $\bar{c}_n(q)$ by slightly adjusting the definition of normalized Genocchi permutations.
In \cite{Feigin2,Feigin}, Feigin introduced a $q$-analog of the normalized median Genocchi number $h_n$ with the Poincaré polynomial $P_{\F_n^a}(q)$ of the degenate flag variety $\F_n^a$ (whose Euler characteristic is $P_{\F_n^a}(1) = h_n$), and gave a combinatorial interpretation of $P_{\F_n^a}(q)$ through Dellac configurations.

\begin{theorem}[Feigin, 2012] \label{theoreminterpretationcombinatoiredepoincare}
For all $n \geq 0$, the polynomial $P_{\F_n^a}(q)$ is generated by $DC(n)$:
$$P_{\F_n^a}(q) = \sum_{C \in DC(n)} q^{2\text{inv}(C)}.$$
\end{theorem}

The degree of the polynomial $P_{\F_n^a}(q)$ being $n(n+1)$ (for algebraic considerations, or because every Dellac configuration $C \in DC(n)$ has at most $\binom{n}{2}$ inversions, see \S \ref{sec:preliminaries}), Feigin introduced the following $q$-analog of $h_n$:
\begin{equation} \label{definitionhnq}
\tilde{h}_n(q) = q^{\binom{n}{2}} P_{\F_n^a}(q^{-1/2}) =  \sum_{C \in DC(n)} q^{\binom{n}{2} - \text{inv}(C)},
\end{equation} 
and proved the following theorem by using the geometry of quiver Grassmannians (see \cite {Feigin3}) and Flajolet's theory of continued fractions \cite{Flajolet}.

\begin{theorem}[Feigin, 2012] \label{theogenerfunctionhn}
The generating function $\sum_{n \geq 0} \tilde{h}_n(q) t^n$ has the continued fraction expansion of Formula (\ref{expansioncn+1}).
\end{theorem}

\begin{corollary}[Feigin, 2012] \label{theohncn}
For all $n \geq 0$, we have $\tilde{h}_n(q) = \bar{c}_{n+1}(q)$.
\end{corollary}

This raises two questions.
\begin{enumerate}
\item Prove combinatorially Corollary \ref{theohncn} by constructing a bijection between Dellac configurations and some appropriate model of $\bar{c}_n(q)$ which preserves the statistics.
\item Prove combinatorially Theorem \ref{theogenerfunctionhn} within the framework of Flajolet's theory of continued fractions by defining a combinatorial model of $\tilde{h}_n(q)$ related to Dyck paths (see \cite{Flajolet}), and constructing a statistic-preserving bijection between Dellac configurations and that new model.
\end{enumerate}

The aim of this paper is to answer above two questions. We answer the first one in \S \ref{sec:dellacdumont}. In \S \ref{sec:preliminaries}, we define a combinatorial model of $\bar{c}_n(q)$ through \textit{normalized Dumont permutations}, and we provide general results about Dellac configurations. In \S \ref{sec:algorithms}, we enounce and prove Theorem \ref{bijectionconfigdumont}, which connects Dellac configurations to normalized Dumont permutations through a stastistic-preserving bijection, and implies immediatly Corollary \ref{theohncn}.\\
We answer the second question in \S \ref{sec:dellacdyck}. In \S \ref{sec:flajolet}, we recall the definition of a Dyck path and some results of Flajolet's theory of continued fractions. In \S \ref{sec:histoires}, we define \textit{Dellac histories}, which consist of Dyck paths weighted with pairs of integers, and we show that their generating function has the continued fraction expansionn of Formula (\ref{expansioncn+1}). In \S \ref{sec:bijections}, we enounce and prove Theorem \ref{bijectiondellacstory}, which connects Dellac configurations to Dellac histories through a statistic-preserving bijection, thence proving Theorem \ref{theogenerfunctionhn} combinatorially.

\section{Connection between Dellac configurations and Dumont permutations}
\label{sec:dellacdumont}

In \S \ref{sec:preliminaries}, we define \textit{normalized Dumont permutations} of order $2n$, whose set is denoted by $\D_n'$, and we prove that they generate $\bar{c}_n(q)$ with respect to the statistic $st$ defined in Formula (\ref{definitionst}), then we define the label of a Dellac configuration and a \textit{switching} transformation on the set $DC(n)$. In \S \ref{sec:algorithms}, we enounce Theorem \ref{bijectionconfigdumont} and we intend to demonstrate it. To do so, we first give two algorithms $\phi :~DC(n) \rightarrow \D_{n+1}'$ and $\varphi :~\D_{n+1}' \rightarrow~DC(n)$, and we prove that $\phi$ and $\varphi_{|\D_{n+1}''}$ are inverse maps.
Then, we show that Equation (\ref{equationstatisticpresetionconfigdumont}) is true for all $C \in DC(n)$, by showing that it is true for some particular $C^0 \in DC(n)$, then by connecting $C^0$ to every other $C \in DC(n)$ thanks to the switching transformation, which happens to preserve Equation (\ref{equationstatisticpresetionconfigdumont}).

\subsection{Preliminaries}
\label{sec:preliminaries}

\subsubsection{\textbf{Combinatorial interpretation of $\bar{c}_n(q)$}.}

\begin{definition} 
A \textit{normalized Dumont permutation} of order $2n$ is a permutation $\sigma \in \D_n$ such that, for all $j \in [n-1]$, the two integers $\sigma^{-1}(2j)$ and $\sigma^{-1}(2j+1)$ have the same parity if and only if $\sigma^{-1}(2j) > \sigma^{-1}(2j+1)$. Let $\D_n' \subset \D_n$ be the set of these permutations.
\end{definition}

\begin{proposition} \label{Cnpolynomegenerateur} For all $n \geq 1$, we have
$\bar{c}_n(q) = \sum_{\sigma \in  \D_n'} q^{st(\sigma)}.$
\end{proposition}

\begin{proof}
Let $j \in [n-1]$ and $\sigma \in \D_n$. Recall that $st(\sigma) = n^2 - \sum_{i=1}^n \sigma(2i) - \text{inv}(\sigma^o) - \text{inv}(\sigma^e)$. It is easy to see that the composition $\sigma' = (2j,2j+1) \circ \sigma$ of $\sigma$ with the transposition $(2j,2j+1)$ is still a Dumont permutation, and that if $\sigma$ fits the condition 
$C(j)$ defined as 
$$\sigma^{-1}(2j) > \sigma^{-1}(2j+1) \Leftrightarrow \text{ $\sigma^{-1}(2j)$ and $\sigma^{-1}(2j+1)$ have the same parity},$$
then
$st(\sigma') = st(\sigma) + 1$.
Now, if we denote by $\D_n^j \subset \D_n$ the subset of permutations that fit the condition $C(j)$, then $\D_n$ is the disjoint union 
$\D_n^j \sqcup \left( (2j,2j+1) \circ \D_n^j \right)$, where $(2j,2j+1) \circ \D_n^j$ is the set $\{(2j,2j+1) \circ \sigma, \ \sigma \in \D_n^j\}$. Since $st((2j,2j+1) \circ \sigma) = st(\sigma) + 1$ for all $\sigma \in \D_n^j$, Formula (\ref{formuleCn}) of Theorem \ref{theointerpretationcombinatoireCn1q} becomes
$$C_n(1,q) = (1+q) \sum_{\sigma \in \D_n^j} q^{st(\sigma)}.$$ 
This yields immediatly: 
$$C_n(1,q) = (1+q)^{n-1} \sum_{\sigma \in \bigcap_{j=1}^{n-1} \D_n^j} q^{st(\sigma)} = (1+q)^{n-1} \sum_{\sigma \in \D_n'} q^{st(\sigma)}.$$
The proposition then follows from Formula (\ref{definitioncnq}). \end{proof}

\subsubsection{\textbf{Label of a Dellac configuration}}

\begin{definition} \label{defetiquetage}
Let $C \in DC(n)$. For all $i \in [n]$, the dot of the $i$-th line of $C$ (from bottom to top) is labeled by the integer $e_i = 2i+2$, and the dot of the $(n+i)$-th line is labeled by the integer $e_{n+i} = 2i-1$ (see Figure \ref{Etiquetage} for an example).\\
\begin{figure}[!h] \centering 
\includegraphics[width=4cm]{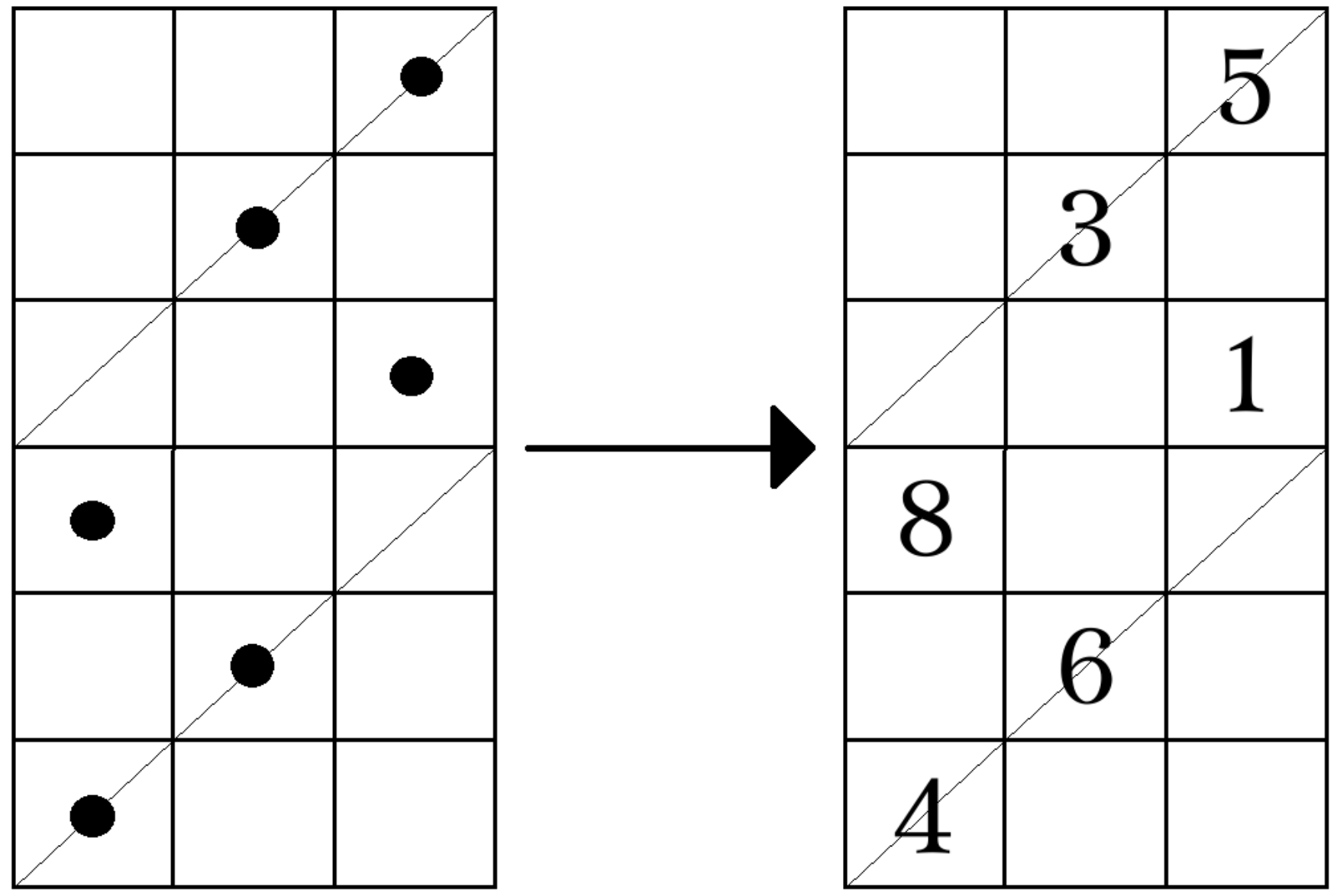} 
\caption{Label of a Dellac configuration $C \in DC(3)$.} 
\label{Etiquetage}
\end{figure} 
>From now on, we will assimilate each dot of a Dellac configuration into its label.
\end{definition}

\begin{definition}[Particular dots] \label{defparticulardots} Let $C \in DC(n)$. For all $j \in [n]$, we define $i_1^C(j) < i_2^C(j)$ such that the two dots of the $j$-th column of $C$ (from left to right) are $e_{i_1^C(j)}$ and $e_{i_2^C(j)}$. When there is no ambiguity, we write $e_{i_1(j)}$ and $e_{i_2(j)}$ instead of $e_{i_1^C(j)}$ and $e_{i_2^C(j)}$.\\
Finally, for all $i \in [n]$, we define the integers $p_C(i)$ and $q_C(i)$ such that $e_{p_C(i)}$ and $e_{n+q_C(i)}$ are respectively the $i$-th even dot and $i$-th odd dot of the sequence $$\left( e_{i_1(1)}, e_{i_2(1)}, e_{i_1(2)}, e_{i_2(2)}, \hdots, e_{i_1(n)}, e_{i_2(n)} \right).$$
For example, in Figure \ref{Etiquetage}, we have  $(e_{i_1(2)}, e_{i_2(2)}) = (6,3) = (e_2,e_5) = (e_{p_C(3)},e_{3+q_C(1)})$.
\end{definition}

\begin{remark} \label{limitationsjetons}
For all $i \in [2n]$, if the dot $e_i$ appears in the $j_i$-th column of $C$, then, by Definition \ref{definitiondellacconfiguration}, we have $j_i \leq i \leq j_i+n$.
As a result, the first $j$ columns of $C$ always contain the $j$ even dots $$e_1,e_2,\hdots,e_j,$$ and the only odd dots they may contain are $$e_{n+1},e_{n+2},\hdots,e_{n+j}.$$
Likewise, the last $n-j+1$ columns of $C$ always contain the $n-j+1$ odd dots
$$e_{n+j},e_{n+j+1},\hdots,e_{2n},$$ and the only even dots they may contain are $$e_j, e_{j+1},e_{j+2},\hdots,e_n.$$
\end{remark}

\begin{remark} \label{remarqueinegalitejetons}

Let $C \in DC(n)$ and $j \in [n]$. If the $j$-th column of $C$ contains the even dot $e_{i \leq n} = 2i+2$, then, since $j \leq i$, we have $e_i \in \{2j+2,2j+4,\hdots,2n+2\}$. Similarly, if the $j$-th column of $C$ contains the odd dot $e_{i > n} = 2(i-n)-1$, since $i \leq j+n$, we have $e_i \in \{1,3,\hdots, 2j-1\}$. As a result, we obtain the following equivalences:
$$e_{i_1^C(j)} > e_{i_2^C(j)} \Leftrightarrow i_1^C(j) \leq n < i_2^C(j) \Leftrightarrow \text{$e_{i_1^C(j)}$ and $e_{i_2^C(j)}$ have different parities.}$$

\end{remark}  

\begin{definition}[Particular configurations] \label{Sigma0}
For all $n \geq 1$, we denote by $C_0(n)$ (respectively $C_1(n)$) the Dellac configuration of size $n$ such that $(e_{i_1(j)},e_{i_2(j)}) = (e_{2j-1},e_{2j})$ (resp. $(e_{i_1(j)},e_{i_2(j)}) = (e_{j},e_{n+j})$) for all $j \in [n]$.
For example $C_0(3)$ (on the left) and $C_1(3)$ (on the right) are the two configurations depicted in Figure \ref{coc1}.
\end{definition}
It is obvious that $C_0(n)$ is the unique Dellac configuration of size $n$ with $0$ inversion, and that $\text{inv}(C_1(n)) = \binom{n}{2}$. We can also prove by induction on $n \geq 1$ that every Dellac configuration $C \in DC(n)$ has at most $\binom{n}{2}$ inversions with equality if and only if $C = C_1(n)$.

\begin{figure}[!h]
\centering
\begin{minipage}{.5\textwidth}
  \centering
\includegraphics[width=4cm]{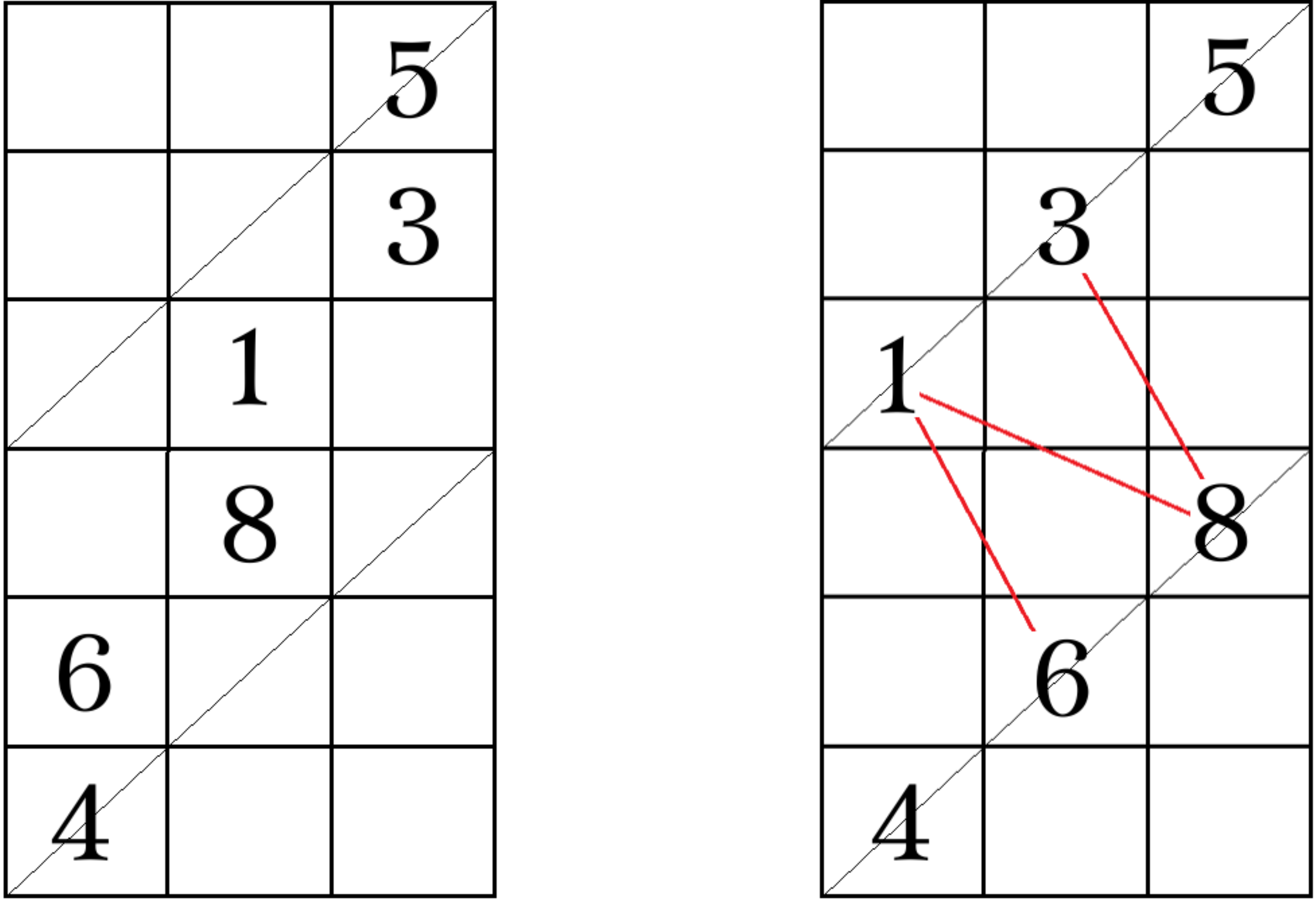}
  \caption{On the left: $C_0(3)$; on the right: $C_1(3)$.}
\label{coc1}
\end{minipage}%
\begin{minipage}{.5\textwidth}
  \centering
\includegraphics[width=3.7cm]{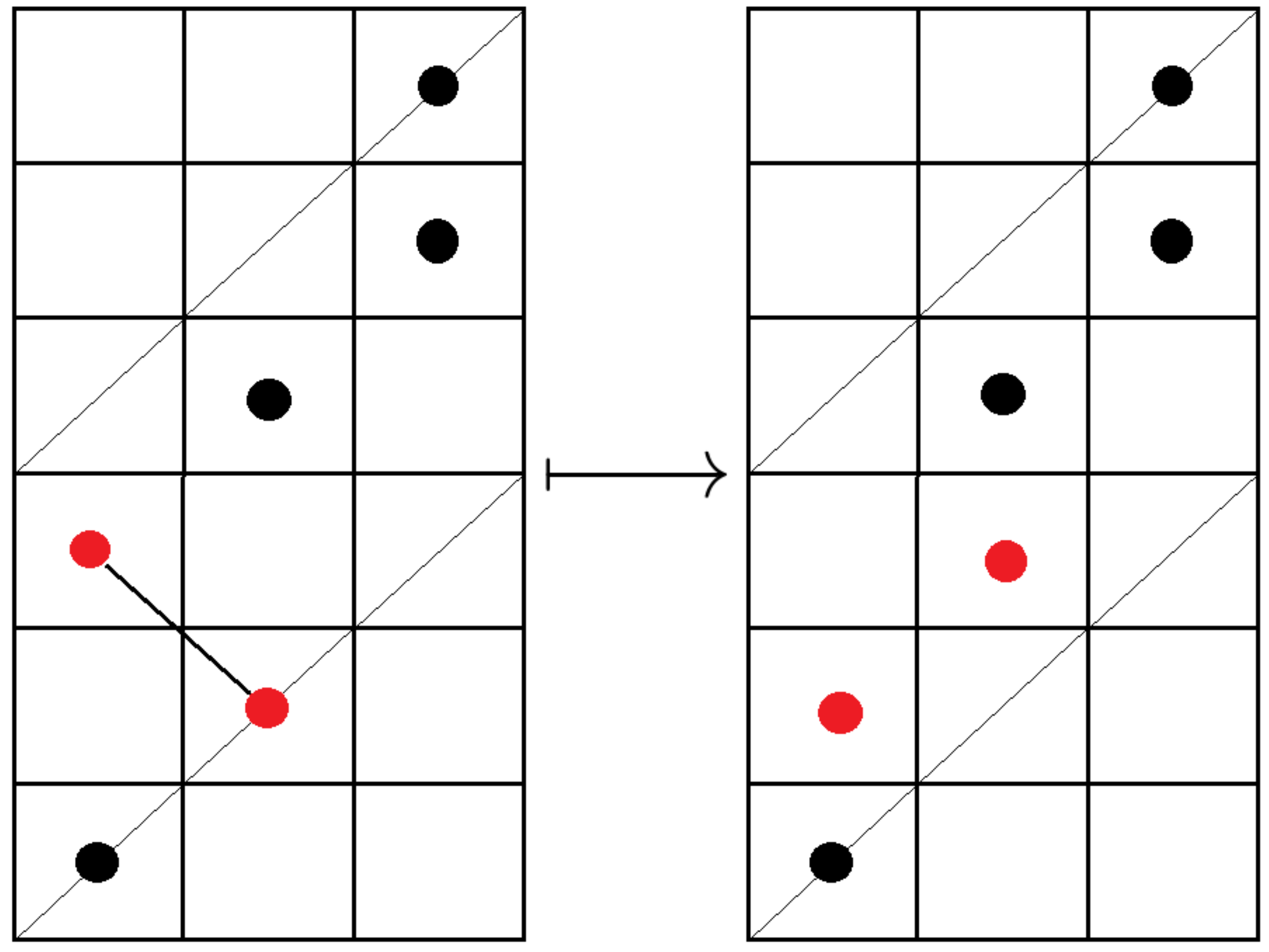} 
  \caption{The Dellac configuration $C \in DC(3)$ is mapped to $Sw^2(C)\in DC(3)$.}
\label{Pivot}
\end{minipage}
\end{figure}

\subsubsection{\textbf{Refinements of the inv statistic on $DC(n)$}}

\begin{definition}
Let $C \in DC(n)$ and $i \in [2n]$.
We define the quantity $l_C(e_i)$ (resp. $r_C(e_i)$) as the number of inversions of $C$ between the dot $e_i$ and any dot $e_{i'}$ with $i' >i$ (resp. $i' < i$).
For example, if $C = C_1(3)$ (see Figure \ref{coc1}), then $l_C(6) = r_C(3) = 1$ and $r_C(1) = l_C(8) = 2$.
\end{definition}

\subsubsection{\textbf{Switching of a Dellac configuration}} 

In the following definition, we provide a tool which transforms a Dellac configuration of $DC(n)$ into a slightly modified tableau, which is not necessarily a Dellac configuration and consequently brings the notion of \textit{switchability}.

\begin{definition} \label{switching}
Let $C \in DC(n)$ and $i \in [2n-1]$.
We denote by $Sw^{i}(C)$ the tableau obtained by switching the two consecutive dots $e_i$ and $e_{i+1}$ (\textit{i.e.}, inserting $e_i$ in $e_{i+1}$'s column and $e_{i+1}$ in $e_i$'s column). 
If the tableau $Sw^i(C)$ is still a Dellac configuration, we say that $C$ is \textit{switchable} at $i$.
In Figure \ref{Pivot}, we give an example $C \in DC(3)$ switchable at $2$.
\end{definition}

It is easy to verify the following assertions.

\begin{fact} \label{facswitchinv}
If $C \in DC(n)$ is switchable at $i$, then $|\text{inv}((Sw^i(C))) - \text{inv}(C)| \leq 1 .$
\end{fact}

\begin{fact} \label{factswicthcond}
A Dellac configuration $C \in DC(n)$ is switchable at $i \in [2n-1]$ if and only if $C$ and $i$ satisfy one of the two following conditions:
\\ \ \\
$(1)$ $i \leq n$ and if $e_{i+1}$ is in the $j_{i+1}$-th column of $C$, then
$j_{i+1} < i+1.$\\
$(2)$ $i > n$ and if $e_i$ is in the $j_i$-th column of $C$, then
$j_i > i-n.$
\end{fact}

In particular :

\begin{fact} \label{factinvolution}
If $C$ is switchable at $i$, then $Sw^i(C)$ is still switchable at $i$ and $Sw^i(Sw^i(C)) = C$.
\end{fact}

\begin{fact} \label{factmemecolonne}
If $e_i$ and $e_{i+1}$ are in the same column of $C$, then $C$ is switchable at $i$ and $C = Sw^i(C)$.
\end{fact}

\begin{fact} \label{factinversion}
If $(e_i,e_{i+1})$ is an inversion of $C$, then $C$ is switchable at $i$ and $\text{inv}(Sw^i(C)) = \text{inv}(C) - 1$ (like in Figure \ref{Pivot}).
\end{fact}

\begin{fact} \label{facttoujoursenn}
A Dellac configuration $C \in DC(n)$ is always switchable at $n$.
\end{fact}

\subsection{Construction of a statistic-preserving bijection}
\label{sec:algorithms}

In this part, we intend to prove the following result.

\begin{theorem} \label{bijectionconfigdumont}
There exists a bijection $\phi : DC(n) \rightarrow \D_{n+1}'$ such that the equality
\begin{equation} \label{equationstatisticpresetionconfigdumont}
st(\phi(C)) = \binom{n}{2} - \text{inv}(C)
\end{equation}
is true for all $C \in DC(n)$.
\end{theorem}

In the following, we define $\phi : DC(n) \rightarrow \D_{n+1}'$ and in order to prove that it is bijective, we construct $\varphi : \D_{n+1} \rightarrow DC(n)$ such that $\phi$ and $\varphi_{|\D_{n+1}'}$ are inverse maps.

\subsubsection{\textbf{Algorithms}}
\textbf{Definition of $\phi$.} We define $\phi : DC(n) \rightarrow \Sig_{2n+2}$ by mapping $C \in DC(n)$ to the permutation $\phi(C) \in \Sig_{2n+2}$ defined as the inverse map of the permutation 
$$2  \widehat{e_{i_2(1)} e_{i_1(1)}} \widehat{e_{i_2(2)} e_{i_1(2)}}  \hdots \widehat{e_{i_2(n)} e_{i_1(n)}} (2n+1),$$
where we recall that $e_{i_1(j)}$ and $e_{i_2(j)}$ are respectively the lower and upper dots of the $j$-th column of $C$ for all $j \in [n]$.

\begin{example} \label{exemplevarphiC}
If $C \in DC(3)$ is the Dellac configuration depicted in Figure \ref{exemplevarphisigma}, we obtain $\phi(C)^{-1} = 2 \widehat{84}  \widehat{16} \widehat{53} 7$.\\
\begin{figure}[!h] \center
\includegraphics[width=1.5cm]{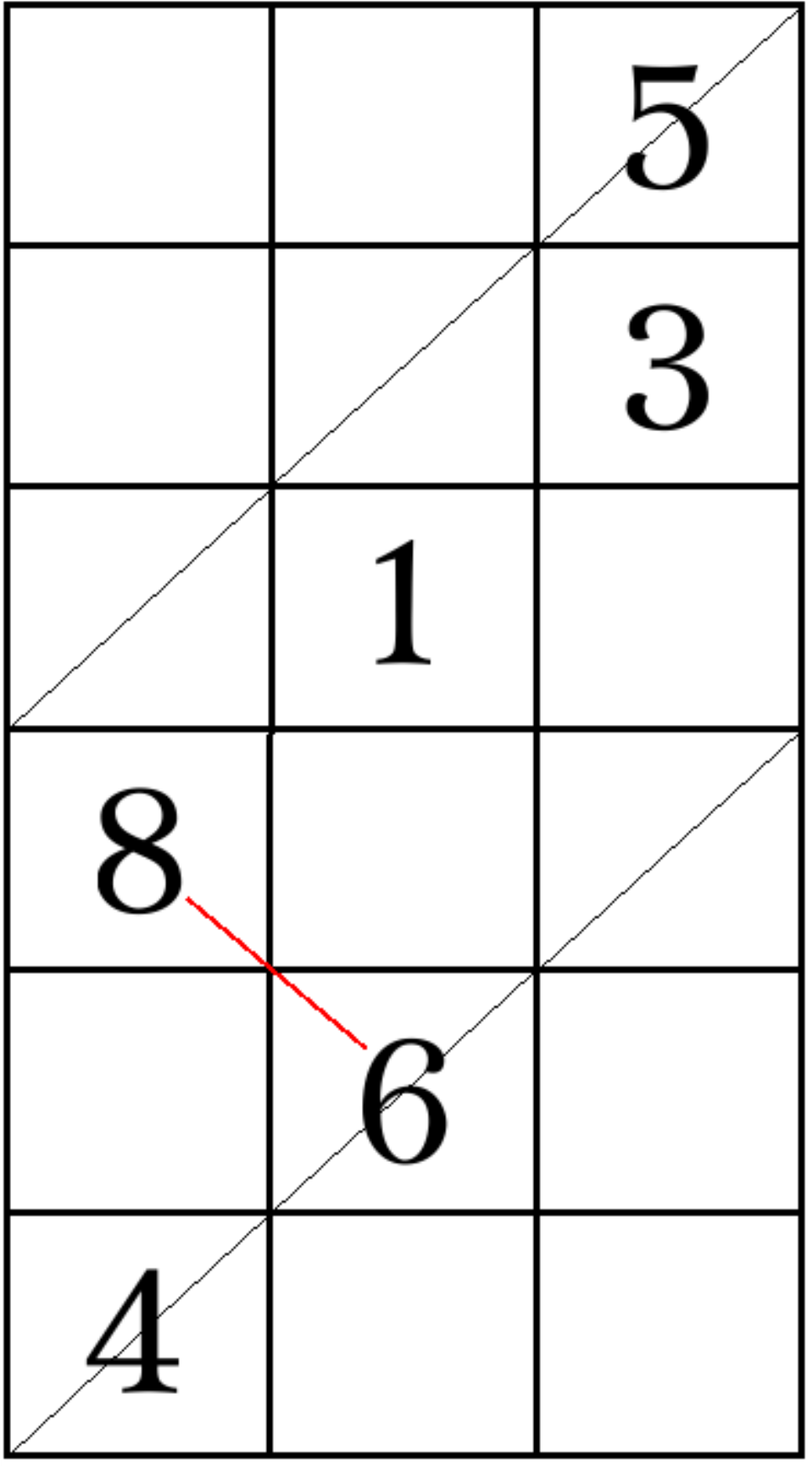}
\caption{$C \in DC(3)$.}
\label{exemplevarphisigma}
\end{figure}
\end{example}

\begin{proposition} \label{varphiCDumont}
For all $C \in DC(n)$, the permutation $\phi(C)$ is a normalized Dumont permutation.
\end{proposition}
\begin{proof} Let $\sigma$ be $\phi(C)$.
It is a Dumont permutation : $(\sigma(2),\sigma(2n+1)) = (1,2n+2)$ and for all $i \in \{2,3,\hdots,n-1\}$, if the dot $2i = e_{i-1}$ is in the $j$-th column of $C$ (resp. if the dot $2i+1 = e_{n+1+i}$ is in the $j'$-th column of $C$), then $\sigma(2i) = \sigma(e_{i-1}) \leq 2j+1 < 2i$ because $j \leq i-1$ (resp. $\sigma(2i+1) = \sigma(e_{n+1+i}) \geq 2j' > 2i+1$ because $n+1+i \leq j'+n$).
It is also normalized according to Remark \ref{remarqueinegalitejetons}.
\end{proof}
\hspace*{-5.9mm}
\textbf{Definition of $\varphi$.} Let $\T_n$ be the set of tableaux of size $n \times 2n$ whose each row contains one dot and each column contains two dots. We define $\varphi : \D_{n+1} \rightarrow \T_n$ by mapping $\sigma \in \D_{n+1}$ to the tableau $\varphi(\sigma) \in \T_n$ whose $j$-th column contains the two dots labelled by $\sigma^{-1}(2j)$ and $\sigma^{-1}(2j+1)$ for all $j \in [n]$.

\begin{proposition} For all $\sigma \in \D_{n+1}$, the tableau $\varphi(\sigma)$ is a Dellac configuration. \end{proposition}

\begin{proof}
Let $j \in [n]$ and $i \in [2n]$ such that $\varphi(\sigma)$ contains a dot in the box $(j,i)$ (\textit{i.e.}, the $j$-th column of $\varphi(\sigma)$ contains the dot $e_i$). By definition $2j \leq \sigma(e_i) \leq 2j+1$. If $i \leq n$, then $e_i = 2i+2$ and $2j \leq \sigma(2i+2) < 2i+2$ thence $j \leq i < j+n$. Else $e_i = 2(i-n)-1$ and $2j+1 \geq \sigma(2(i-n)-1) > 2(i-n)-1$ thence $j \geq i-n > 0 \geq j-n$. In either case we obtain $j \leq i \leq j+n$ so $\varphi(\sigma) \in DC(n)$.
\end{proof}

\begin{example} \label{exemplephisigma}
Consider the permutation $\sigma = 41726583 \in \D_4$. From $\sigma^{-1} = 2 \widehat{48} \widehat{16}  \widehat{53} 7$, we obtain the Dellac configuration $\varphi(\sigma)$ illustrated in Figure \ref{exemplevarphisigma}.
\end{example}

\hspace*{-5.9mm} 
It is easy to verify that $\phi \circ \varphi_{|\D_{n+1}'} = Id_{\D_{n+1}'}$ and $\varphi \circ \phi = Id_{DC(n)}$.

\begin{remark} \label{remactiondegroupe}
There is a natural interpretation in terms of group action : in the proof of Proposition \ref{Cnpolynomegenerateur}, we show that the subgroup of $\Sig_{2n+2}$ generated by the $n$ permutations $(2,3)$, $(4,5)$,~..., $(2n,2n+1)$, freely operates by left multiplication on $\D_{n+1}$, and that each orbit of that action contains exactly one normalized Dumont permutation.
Also, the orbits are indexed by elements of $DC(n)$ : two permutations $\sigma_1$ and $\sigma_2 \in \D_{n+1}$ are in the same orbit if and only if $\varphi(\sigma_1) = \varphi(\sigma_2)$, and for all $\sigma \in \D_{n+1}$, the permutation $\phi(\varphi(\sigma))$ is the unique normalized Dumont permutation in the orbit of $\sigma$.
\end{remark}

\begin{example} 
In Examples \ref{exemplevarphiC} and \ref{exemplephisigma}, we have $\varphi(\phi(C)) = C$ and $\phi(\varphi(\sigma)) = (2,3) \circ \sigma$.
\end{example}

\subsubsection{\textbf{Alternative algorithm}}

\begin{definition}
Let $(y_1,y_2,\hdots,y_{2n})$ be the sequence $(3,2,5,4,\hdots,2n+1,2n)$.
For all $C \in DC(n)$, we define the permutation $\tau_C \in \Sig_{2n}$ by $\phi(C)(e_i) = y_{\tau_C(i)}$ for all $i \in [2n]$.
\end{definition}

\begin{lemma} \label{equivinversionconfigtau}
Let $C \in DC(n)$ and $(p,q) \in [2n]^2$ such that $p<q$. Then $(e_p,e_q)$ is an inversion of $C$ if and only if $(p,q)$ is an inversion of $\tau_C$, \textit{i.e.}, if $\tau_C(p) > \tau_C(q)$.
\end{lemma}

\begin{proof} Recall that if the dot $e_i$ is located in the $j$-th column of $C$, then $\phi(C)(e_i) = 2j$ or $2j+1$. Consequently, since $y_i =i$ if $i$ is even, and $y_i = i+2$ if $i$ is odd, then $\tau_C(i)=2j$ or $2j-1$.
Now let $1 \leq p < q \leq 2n$, and let $(j_p,j_q)$ such that the dot $e_p$ (resp. $e_q$) is located in the $j_p$-th column (resp. $j_q$-th column) of $C$. If $(e_p,e_q)$ is an inversion of $C$, \textit{i.e.}, if $j_p > j_q$, then $\tau_C(p) \geq 2j_p - 1 > 2j_q \geq \tau_C(q)$ and $(p,q)$ is an inversion of $\tau_C$.
Reciprocally, if $\tau_C(p) > \tau_C(q)$, then $2j_p \geq  \tau_C(p) > \tau_C(q) \geq 2j_q-1$,
hence $j_p \geq j_q.$
Now suppose that $j_p = j_q =: j$. It means that $e_p$ and $e_q$ are the lower dot and the upper dot of the $j$-th column respectively, which translates into
$y_{\tau_C(p)} = \phi(C)(e_p) = 2j+1$ and $y_{\tau_C(q)} = \phi(C)(e_q) = 2j.$ Consequently, we obtain $\tau_C(p) = 2j-1$ and $\tau_C(q) = 2j$,
which is in contradiction with $\tau_C(p) > \tau_C(q)$.
So $j_p > j_q$ and
$(e_p,e_q)$ is an inversion of $C$.
\end{proof}

\begin{proposition}[Alternative algorithm $\phi : DC(n) \rightarrow \D_{n+1}'$] \label{lemdirect}
Let $C \in DC(n)$. For all $i \in [2n]$, we have $\tau_C(i) = i +l_C(e_i)-r_C(e_i)$.
\end{proposition}

\begin{example}
Consider the following Dellac configuration $C \in DC(3)$.
$$\includegraphics[width=2.5cm]{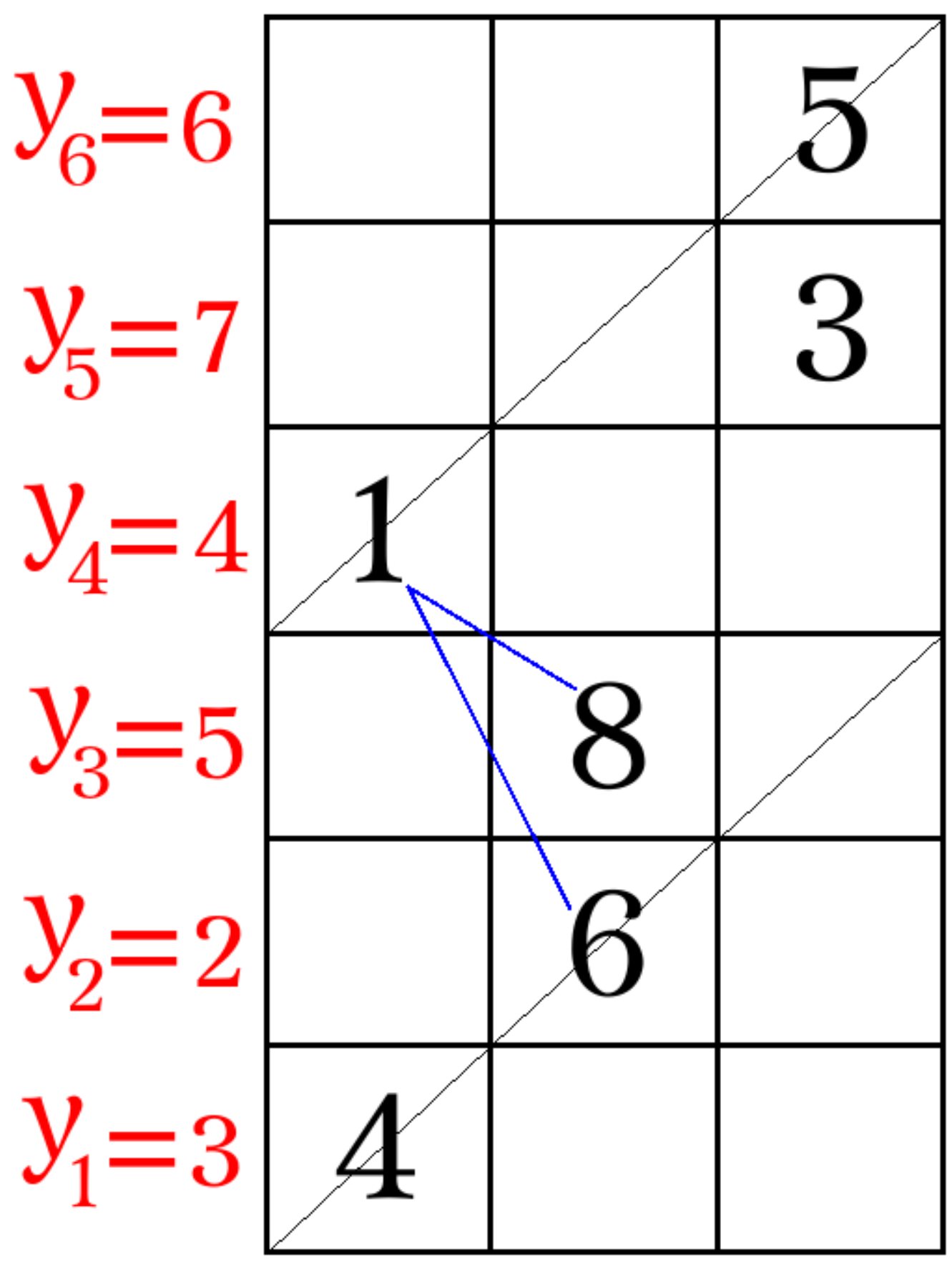}$$
By Proposition  \ref{lemdirect}, we obtain immediatly $\phi(C) = 21736584$. This is coherent with the algorithm given in Definition \ref{definitionphiC}, which says that $\phi(C)^{-1} = 2 \widehat{14}  \widehat{86}  \widehat{53} 7$.
\end{example}

\hspace*{-5.9mm}
\textbf{Proof of Lemma \ref{lemdirect}.}
>From Lemma \ref{equivinversionconfigtau}, we know that
$$\begin{cases} l_C(e_i) \hspace*{0.5mm} = | \{k > i \ | \ \tau_C(k) < \tau_C(i) \}|,\\
r_C(e_i) = |\{k < i \ | \ \tau_C(k) > \tau_C(i) \}|. \end{cases}$$
So, the lemma follows from the well-known equality
$$\pi(i) = i + | \{k > i \ | \ \pi(k) < \pi(i) \}| - |\{k < i \ | \ \pi(k) > \pi(i) \}|$$
for all permutation $\pi \in \Sig_m$ and for all integer $m \geq 1$. \hfill $\qed$

\subsubsection{\textbf{Switchability and Dumont permutations}}

We have built a bijection $\phi : DC(n) \rightarrow \D_{n+1}'$. To demonstrate Formula \ref{equationstatisticpresetionconfigdumont}, we will use the notion of switchability defined in \S \ref{sec:preliminaries}, by showing that if Formula \ref{equationstatisticpresetionconfigdumont} is true for some particuliar configuration $C^0$, and if $C^1$ is a configuration connected to $C^0$ by a switching transformation, then Formula \ref{equationstatisticpresetionconfigdumont} is also true for $C^1$. We will also need Lemma \ref{lemswitch1} and Proposition \ref{lemswitch2} to prove (in Proposition \ref{iterswitch}) that any two Dellac configurations are connected by a sequence of switching transformations.

\begin{lemma} \label{lemswitch1}
Let $\sigma \in \D_{n+1}$ and $i \in [2n-1]$. We denote by $\sigma'$ the composition $\sigma \circ (e_i,e_{i+1})$ of the transposition $(e_i,e_{i+1})$ with the permutation $\sigma$.
The Dellac configuration $\varphi(\sigma)$ is switchable at $i$ if and only if $\sigma'$ is still a Dumont permutation, and in that case $\varphi(\sigma') = Sw^i(\varphi(\sigma))$.
\end{lemma}

\begin{proof}
Let $T$ be the tableau $Sw^i(\varphi(\sigma))$. If $T$ is a Dellac configuration, one can check that $\sigma' \in \D_{n+1}$ thanks to Fact \ref{factswicthcond}.
Reciprocally, if $\sigma'$ is a Dumont permutation, we may consider the Dellac configuration $\varphi(\sigma')$. For all $j \in [n]$, let $\left( e_{i_1(j)}, e_{i_2(j)} \right)$ (with $i_1(j) < i_2(j)$) be the two dots of the $j$-th column of $\varphi(\sigma)$, and $\left( e_{i_1'(j)}, e_{i_2'(j)} \right)$ (with $i_1'(j) < i_2'(j)$) the two dots of the $j$-th column of $\varphi(\sigma')$. 
Then $e_{i_1'(j)} = \sigma'^{-1}(2j+1) = (e_i,e_{i+1}) \circ \sigma^{-1}(2j+1) = (e_i,e_{i+1}) \left( e_{i_1(j)}  \right)$ and $e_{i_2'(j)} = \sigma'^{-1}(2j) = (e_i,e_{i+1}) \circ \sigma^{-1}(2j) = (e_i,e_{i+1}) \left( e_{i_2(j)} \right)$
for all $j$, which exactly translates into $\varphi(\sigma') = Sw^i(\varphi(\sigma)) = T$.
\end{proof}

\hspace*{-5.9mm} 
The following result is easy.

\begin{proposition} \label{lemswitch2} In the setting of Lemma \ref{lemswitch1}, if $\varphi(\sigma)$ is switchable at $i$, then the following propositions are equivalent.
\begin{enumerate}
\item $\varphi(\sigma') \neq \varphi(\sigma)$;
\item the two dots $e_i$ and $e_{i+1}$ are not in the same column of $\varphi(\sigma)$;
\item $\text{inv}(\varphi(\sigma')) - \text{inv}(\varphi(\sigma)) = \pm 1$;
\item $\phi(\varphi(\sigma)) \circ (e_i,e_{i+1}) \in \D_{n+1}'$;
\item $\phi(\varphi(\sigma')) = \phi(\varphi(\sigma)) \circ (e_i,e_{i+1})$.
\end{enumerate}
\end{proposition}

\begin{proposition} \label{iterswitch}
Let $(C_1,C_2) \in DC(n)^2$. There exists a finite sequence of switching transformations from $C_1$ to $C_2$, \textit{i.e.}, a sequence $(C^0, C^1, \hdots, C^m)$ in $DC(n)$ for some $m \geq 0$ such that $(C^0,C^m) = (C_1,C_2)$ and such that $C^{k} = Sw^{i_{k-1}}(C^{k-1})$ for some index $i_{k-1} \in [2n]$, for all $k \in [m]$.
\end{proposition}

\begin{proof} From Fact \ref{factinvolution}, it is sufficient to prove that for all $C \in DC(n)$, there exists a finite sequence of switching transformations from $C$ to $C_0(n)$, the unique Dellac configuration of size $n$ with $0$ inversion (see Definition \ref{Sigma0}).
If $C = C_0(n)$, the statement is obvious. Else, let $C^0 = C$.
>From Lemma \ref{equivinversionconfigtau}, for all $i \in [2n]$, the pair $(e_i,e_{i+1})$ is an inversion of $C^0$ if and only if the integer $i$ is a descent of $\tau_{C^0}$, \textit{i.e.}, if $\tau_{C^0}(i) > \tau_{C^0}(i+1)$.
Now, from Proposition \ref{lemdirect}, the permutation $\tau_{C_0(n)}$ is the identity map $Id$ of $\Sig_{2n+2}$. Consequently, since $C^0 \neq C_0(n)$, we have $\tau_{C^0} \neq Id_{\Sig_{2n}}$, so $\tau_{C^0}$ has at least one descent.
Let $i_0$ be one of those descents, and let $C^1 = Sw^{i_0}(C^0) \in DC(n).$
Since $(e_{i_0},e_{i_0+1})$ is an inversion of $C^0$, in particular $e_{i_0}$ and $e_{i_0+1}$ are not in the same column, so, from Proposition \ref{lemswitch2}, we have $\phi(C^1) = \phi(C^0) \circ (e_{i_0}, e_{i_0+1})$,
hence
$\tau_{C^1} = \tau_{C^0} \circ (i_0,i_0+1).$
Consequently, since $i_0$ is a descent of $\tau_{C^0}$, it is not a  descent of $\tau_{C^1}$.
Iterating the process with $C^1$, and by induction, we build a finite sequence of switching transformations $(C^0,C^1, \hdots, C^m)$ such that $\tau_{C^m}$ has no descent, \textit{i.e.}, such that $\tau_{C^m} = Id = \tau_{C_0(n)}$, which implies $C^m = C_0(n)$.
\end{proof}

\begin{example} 
In Figure \ref{diagrammeswitching}, we give a graph whose vertices are the $h_3 = 7$ elements of $DC(3)$, and in which two Dellac configurations are connected by an edge if they are connected by a switching transformation.
\begin{figure}[!h] \center
\includegraphics[width=7.2cm]{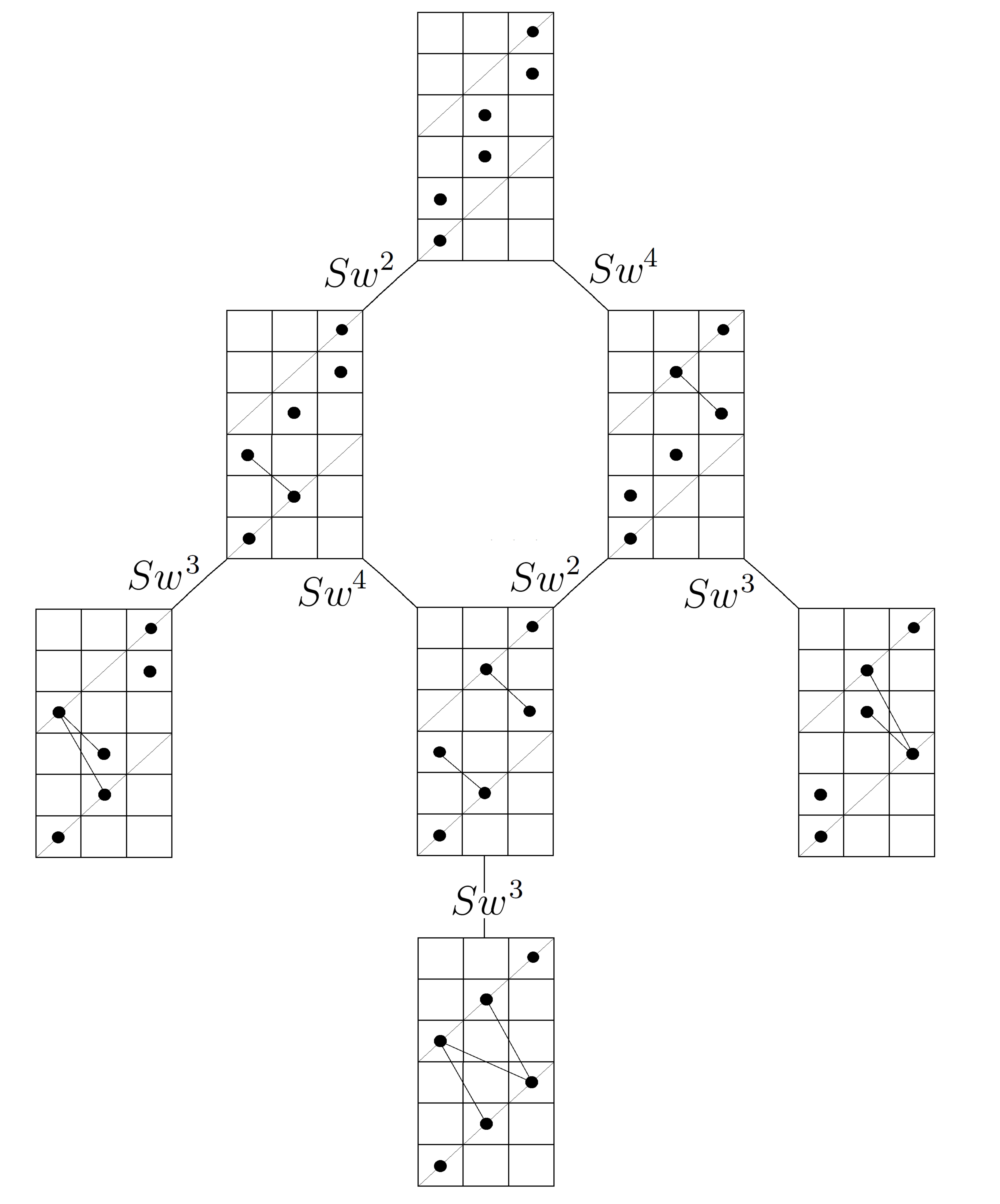}
\caption{The switching transformations of $DC(3)$.}
\label{diagrammeswitching}
\end{figure}
\end{example}

\subsubsection{\textbf{Proof of the statistic preservation formula (\ref{equationstatisticpresetionconfigdumont})}}

We are going to prove that Formula (\ref{equationstatisticpresetionconfigdumont}) is true for all $C \in DC(n)$, which will achieve the proof of Theorem \ref{bijectionconfigdumont}. First notice that it is true for $C = C_1(n)$, the unique Dellac configuration with 
$\binom{n}{2}$ inversions (see Definition \ref{Sigma0}): indeed $\phi(C_1(n))$ is the involution $214365 \hdots (2n+2) (2n+1)$, consequently the two permutations $\phi(C_1(n))^e = 135\hdots (2n+1)$ and $\phi(C_1(n))^o = 246\hdots (2n+2)$ have no inversion, hence $$st(\phi(C_1(n))) = (n+1)^2 - (1 + 3 + 5 + \hdots + (2n+1)) = 0.$$
Let $C \in DC(n)$. From Lemma \ref{iterswitch}, there exists a finite sequence of switching transformations $(C^0, C^1, \hdots, C^m)$ from $C^0 = C_1(n)$ to $C^m = C$.
For all $k \in \{0,1,\hdots,m-1\}$, let $i_{k} \in [2n]$ such that $C^{k+1} = Sw^{i_{k}}(C^{k})$.
We can suppose that $C_{k+1} \neq C_k$, \textit{i.e.}, that $\text{inv}(C^{k+1}) = \text{inv}(C^k) \pm 1.$
Since Formula (\ref{equationstatisticpresetionconfigdumont}) is true for $C_1(n)$, it will be true for $C$ by induction if we show that 
$$st(\phi(C^{k+1})) -st(\phi(C^k)) = \text{inv}(C^{k}) - \text{inv}(C^{k+1})$$ for all $k$.
We know that the quantity $\text{inv}(C^{k}) - \text{inv}(C^{k+1})$ equals $\pm 1$. From Fact \ref{factinvolution}, we have $Sw^{i_k}(C^{k+1}) = C^k$. Then, provided that $C^k$ is replaced by $Sw^{i_k}(C^k) = C^{k+1}$, we can assume that the quantity $\text{inv}(C^{k}) - \text{inv}(C^{k+1})$ equals $1$, which means the pair $(e_{i_k},e_{i_{k+1}})$ is an inversion of $C^k$. Consequently, to achieve the proof of Theorem \ref{bijectionconfigdumont}, it suffices to prove the equality
\begin{equation} \label{equalitytofinishtheproof}
st(\phi(C^{k+1})) -st(\phi(C^k)) = 1
\end{equation}
under the hypothesis $\text{inv}(C^{k}) - \text{inv}(C^{k+1}) = 1$.
Let $\sigma_k = \phi(C^k)$ and $\sigma_{k+1} = \phi(C^{k+1})$. Since $e_{i_k}$ and $e_{i_k+1}$ are not in the same column of $C^k$, we have $\sigma_{k+1} = \sigma_k \circ (e_{i_k},e_{i_k+1})$ in view of Proposition \ref{lemswitch2}.
\\ \ \\
$(a)$ If $e_{i_k}$ and $e_{i_k+1}$ have the same parity (which is always true except for $i_k = n$), then the two integers $e_{i_k}$ and $e_{i_k+1}$ appear in the same subset $\{1,3,\hdots,2n+1\}$ or $\{2,4,\hdots,2n+~2\}$. Consequently, we obtain the two equalities 
\begin{align*}
\sum_{i=1}^{n+1} \sigma_{k+1}(2i) &= \sum_{i=1}^{n+1} \sigma_{k}(2i),\\
(\text{inv}(\sigma_{k+1}^e) - \text{inv}(\sigma_{k}^e), \text{inv}(\sigma_{k+1}^o) - \text{inv}(\sigma_{k}^o)) &= (-1,0) \text{ or $(0,-1)$},
\end{align*} 
thence $st(\sigma_{k+1}) = st(\sigma_k) + 1$, which brings Equality (\ref{equalitytofinishtheproof}).
\\ \ \\
$(b)$ Else $i_k = n$ and $(e_{i_k},e_{i_k+1}) = (2n+2,1)$.
>From $\sigma_{k+1} = \sigma_k \circ (e_{i_k},e_{i_k+1})$, we obtain
\begin{align*}
\sigma_{k+1}^e & = \sigma_k(2) \sigma_k(4) \hdots \sigma_k(2n) \sigma_k(1),\\
\sigma_{k+1}^o & = \sigma_k(2n+2) \sigma_k(3) \sigma_k(5) \hdots \sigma_k(2n+1).
\end{align*}
This provides the three following equations.
\begin{equation} \label{formulesdenertsigmak+1a}
\sum_{i=1}^{n+1} \sigma_{k+1}(2i) =  \left( \sum_{i=1}^{n+1} \sigma_{k}(2i) \right) - \sigma_k(2n+2) + \sigma_k(1),
\end{equation}
\begin{multline} \label{formulesdenertsigmak+1b}
\text{inv}(\sigma_{k+1}^e) =  \text{inv}(\sigma_k^e)
-|\{2i < 2n+2 \ | \ \sigma_k(2i) > \sigma_k(2n+2) \}|\\+ |\{2i < 2n+2 \ | \ \sigma_k(2i) > \sigma_k(1) \}|,
\end{multline}
\begin{multline} \label{formulesdenertsigmak+1c}
\text{inv}(\sigma_{k+1}^o) =  \text{inv}(\sigma_k^o)
-|\{ 1 < 2i+1 \ | \sigma_k(2i+1) < \sigma_k(1) \}|\\+|\{1 < 2i+1 \ | \ \sigma_k(2i+1) < \sigma_k(2n+2) \}|.
\end{multline}

We need the following lemma to explicit Equalities (\ref{formulesdenertsigmak+1b}) and (\ref{formulesdenertsigmak+1c}).

\begin{lemma} \label{explicitationcardinaux}
We have the equalities
\begin{align}
\label{sigma2isuperieursigma2n+2} |\{2i < 2n+2 \ | \ \sigma_k(2i) > \sigma_k(2n+2) \}| &=r_{C^k}(2n+2) +\left(1 + (-1)^{\sigma_k(2n+2)}\right)/2,\\
\label{sigma2isuperieursigma1} |\{2i < 2n+2 \ | \ \sigma_k(2i) > \sigma_k(1) \}| &=r_{C^k}(1) - \left(1 - (-1)^{\sigma_k(1)}\right)/2,\\
\label{sigma2i+1inferieursigma1} |\{ 1 < 2i+1 \ | \sigma_k(2i+1) < \sigma_k(1) \}| &= l_{C^k}(1) + \left(1-(-1)^{\sigma_k(1)}\right)/2,\\
\label{sigma2i+1inferieursigma2n+2} |\{1 < 2i+1 \ | \ \sigma_k(2i+1) < \sigma_k(2n+2) \}| &= l_{C^k}(2n+2) - \left(1+(-1)^{\sigma_k(2n+2)}\right)/2.
\end{align}
\end{lemma}

\begin{proof}
We only demonstrate Equalities (\ref{sigma2isuperieursigma2n+2}) and (\ref{sigma2isuperieursigma1}), because the proof of (\ref{sigma2i+1inferieursigma1}) is analogous to that of (\ref{sigma2isuperieursigma2n+2}) and the proof of (\ref{sigma2i+1inferieursigma2n+2}) is analogous to that of (\ref{sigma2isuperieursigma1}).
\begin{itemize}
\item
Proof of (\ref{sigma2isuperieursigma2n+2}): if the dot $e_{i_k} = 2n+2$ appears in the $j_k$-th column of $C^k$, and if the dot $e_{i-1} = 2i$ (with $1 \leq i-1 \leq n = i_k$) appears in the $j_{i-1}$-th column of $C^k$, then $\sigma_k(2n+2) \in \{2j_k, 2j_k + 1\}$ and $\sigma_k(2i) \in \{2j_{i-1},2j_{i-1}+1\}$. Consequently, the two following assertions are equivalent:
\begin{itemize}
\item $\sigma_k(2i) > \sigma_k(2n+2)$;
\item either $j_{i-1} > j_k$, or $j_{i-1} = j_k$ and $\sigma_k(2n+2) = 2j_{i-1}$ (which forces $\sigma_k(2i)$ to be $2j_{i-1}+1$).
\end{itemize}
As a result, $$|\{2i < 2n+2 \ | \ \sigma_k(2i) > \sigma_k(2n+2) \}|=r_{C^k}(2n+2) + \delta_{\sigma_k(2n+2)}$$ where $\delta_{\sigma_k(2n+2)} =~1$ if $\sigma_k(2n+2)$ is even, and $\delta_{\sigma_k(2n+2)} =0$ if $\sigma_k(2n+2)$ is odd, \textit{i.e.}, where $\delta_{\sigma_k(2n+2)} = \left(1+(-1)^{\sigma_k(2n+2)}\right)/2.$
\item
Proof of (\ref{sigma2isuperieursigma1}): with the same reasoning as for (\ref{sigma2isuperieursigma2n+2}), we find the equality $$|\{2i < 2n+2 \ | \ \sigma_k(2i) > \sigma_k(1) \}|=r_{C^k}(1) - 1 + \left(1 + (-1)^{\sigma_k(1)}\right)/2$$ (with $r_{C^k}(1) - 1$ instead of $r_{C^k}(1)$ because there is an inversion between $1 = e_{i_{k+1}}$ and $2n+2 = e_{i_k}$, whereas $2n+2$ is not counted in the quantity \\$|\{2i < 2n+2 \ | \ \sigma_k(2i) > \sigma_k(1) \}|$).
Since $- 1 + \left(1 + (-1)^{\sigma_k(1)}\right)/2 = - \left(1 - (-1)^{\sigma_k(1)}\right)/2$, we obtain (\ref{sigma2isuperieursigma1}).
\end{itemize}
\end{proof}

In view of Lemma \ref{explicitationcardinaux}, Equalities (\ref{formulesdenertsigmak+1b}) and (\ref{formulesdenertsigmak+1c}) become

\begin{multline} \label{formulesexpliciteesb}
\text{inv}(\sigma_{k+1}^e) - \text{inv}(\sigma_k^e) =
r_{C^k}(1) - r_{C^k}(2n+2) - 1 + \left((-1)^{\sigma_k(1)}-(-1)^{\sigma_k(2n+2)}\right)/2,
\end{multline}
\begin{multline} \label{formulesexpliciteesc}
\text{inv}(\sigma_{k+1}^o) - \text{inv}(\sigma_k^o) =
l_{C^k}(2n+2) - l_{C^k}(1) - 1 + \left((-1)^{\sigma_k(1)}-(-1)^{\sigma_k(2n+2)}\right)/2.
\end{multline}

Now, from Lemma \ref{lemdirect}, we know that
$$\sigma_k(1) = y_{n+1 + l_{C^k}(1) - r_{C^k}(1)},$$
$$\sigma_k(2n+2) = y_{n+l_{C^k}(2n+2) - r_{C^k}(2n+2)}.$$
>From $y_i = i + 1-(-1)^i$ for all $i$, we deduce the two following formulas.

\begin{equation}
\label{simplificationsigma1} \hspace*{-1.43cm} \sigma_k(1) = n+2 + (-1)^n + l_{C^k}(1) - r_{C^k}(1) + \\ (-1)^{n+1} \left( 1 - (-1)^{l_{C^k}(1) - r_{C^k}(1)} \right),
\end{equation}
\begin{multline}
\label{simplificationsigma2n+2} \sigma_k(2n+2) =n+1 - (-1)^n + l_{C^k}(2n+2) - r_{C^k}(2n+2) \\+ (-1)^n \left( 1 - (-1)^{l_{C^k}(2n+2) - r_{C^k}(2n+2)} \right).
\end{multline}

By injecting Equalities (\ref{simplificationsigma1}) and (\ref{simplificationsigma2n+2}) in Equalities (\ref{formulesdenertsigmak+1a}), (\ref{formulesexpliciteesb}) and (\ref{formulesexpliciteesc}), we obtain the three new equalities

\begin{multline} \label{formulesexplicitees2a}
\sum_{i=1}^{n+1} \sigma_{k+1}(2i) - \sum_{i=1}^{n+1} \sigma_{k}(2i) = 1+l_{C^k}(1) - l_{C^k}(2n+2) + r_{C^k}(2n+2) - r_{C^k}(1)\\ + (-1)^{n+l_{C^k}(1) - r_{C^k(1)}} + (-1)^{n+l_{C^k}(2n+2) - r_{C^k}(2n+2)},
\end{multline}
\begin{multline} \label{formulesexplicitees2b}
\text{inv}(\sigma_{k+1}^e) - \text{inv}(\sigma_k^e) =
r_{C^k}(1) - r_{C^k}(2n+2) - 1 \\- \left( (-1)^{n+l_{C^k}(1) - r_{C^k}(1)}+(-1)^{n + l_{C^k}(2n+2) - r_{C^k}(2n+2} \right) / 2,
\end{multline}
\begin{multline} \label{formulesexplicitees2c}
\text{inv}(\sigma_{k+1}^o) - \text{inv}(\sigma_k^o) =
l_{C^k}(2n+2) - l_{C^k}(1) - 1 \\- \left( (-1)^{n+l_{C^k}(1) - r_{C^k}(1)}+(-1)^{n + l_{C^k}(2n+2) - r_{C^k}(2n+2} \right) / 2.
\end{multline}

Finally, we obtain Equality (\ref{equalitytofinishtheproof}) by summing Equalities (\ref{formulesexplicitees2a}), (\ref{formulesexplicitees2b}) and (\ref{formulesexplicitees2c}). This puts an end to the demonstration of Theorem \ref{bijectionconfigdumont}.

\begin{remark}
In \cite{HZ2}, the authors proved that $\bar{c}_n(q)$ is divisible by $1+q$ if $n$ is odd, but requested a combinatorial proof of this statement. Now, if $n$ is odd, one can prove that every Dellac configuration $C \in DC(n-1)$ is switchable at some even integer, which yields a natural involution $\mathcal{I}$ on $DC(n-1)$ such that inv$(\mathcal{I}(C)) = \text{inv}(C) \pm 1$ for all $C$. This proves combinatorially the divisibility of $\bar{c}_n(q)$ by $1+q$ in view of Theorem \ref{bijectionconfigdumont}.
\end{remark}

\section{Dellac histories}
\label{sec:dellacdyck}

\subsection{Weighted Dyck paths}
\label{sec:flajolet}

Recall (see \cite{Flajolet}) that a \textit{Dyck path} $\gamma$ of length $2n$ is a sequence of points $(p_0,p_1, \hdots, p_{2n})$ in $\N^2$ such that $(p_0,p_{2n})=((0,0),(2n,0))$, and for all $i \in [2n]$, the step $(p_{i-1},p_i)$ is either an \textit{up step} $(1,1)$ or a \textit{down step} $(1,-1)$. We denote by $\Gamma(n)$ the set of Dyck paths of length $2n$.
Furthermore, let $\mu = (\mu_n)_{n \geq 1}$ be a sequence of elements of a ring. A \textit{weighted Dyck path} is a Dyck path $\gamma = (p_i)_{0 \leq i \leq n} \in \Gamma(n)$ whose each up step has been weighted by $1$, and each down step $(p_{i-1},p_i)$ from height $h$ (\textit{i.e.}, such that $p_{i-1} = (i-1,h)$) has been weighted by $\mu_h$.\\
The weight 
\begin{equation} \label{definitionomegamu}
\omega_{\mu}(\gamma)
\end{equation} 
of the weighted Dyck path $\gamma$ is the product of the weights of all steps.

\begin{remark} \label{analysedyckpath}
If $\gamma = (p_i)_{0 \leq i \leq 2n} \in \Gamma(n)$, then $p_i = (i, n_u(i)-n_d(i))$ where $n_u(i)$ and $n_d(i)$ are defined as the numbers of up steps and down steps on the left of $p_i$ respectively (in particular $n_u(i) + n_d(i)= i$). Consequently, since the final point of $\gamma$ is $p_{2n} = (2n,0)$, the path $\gamma$ has exactly $n$ up steps and $n$ down steps, and for all $j \in [n]$, the points $p_{2j-1}$ and $p_{2j}$ are at heights respectively odd and even.
\end{remark}

\begin{definition}[Labelled steps] \label{defithsteps}
Let $\gamma = (p_i)_{0 \leq i \leq 2n} \in \Gamma(n)$. For all $i \in [n]$, we denote by $s^u_i(\gamma)$ (resp. $s^d_i(\gamma)$) the $i$-th up step (resp. down step) of $\gamma$.
When there is no ambiguity, we write $s^u_i$ and $s^d_i$ instead of $s^u_i(\gamma)$ and $s^d_i(\gamma)$.
\end{definition}

\begin{remark} \label{relationijk}
If $s^u_i(\gamma) = (p_{2j-2},p_{2j-1})$ or $(p_{2j-1},p_{2j})$ where $p_{2j-2} = (2j-2,2k)$ for some $k \geq 0$, then, following Remark \ref{analysedyckpath}, we know that $2k = n_u(2j-2) - n_d(2j-2) = 2 n_u(2j-2) - (2j-2)$, and by definition of $s^u_i(\gamma)$ it is necessary that $n_u(2j-2) = i-1$, and we obtain $2k = 2(i-j)$ hence $i = j+k$. In the same context, if $s^d_i(\gamma) = (p_{2j-1},p_{2j})$ or $(p_{2j-2},p_{2j-1})$, then we obtain $i = j-k$ by an analogous reasoning.
\end{remark}

\subsection{Dellac histories}
\label{sec:histoires}

\begin{definition} \label{definitiondellacstory}
A \textit{Dellac history} of length $2n$ is a pair $(\gamma, \xi)$ where $\gamma = (p_i)_{0 \leq i \leq 2n} \in~\Gamma(n)$ and $\xi = (\xi_1,\xi_2, \hdots, \xi_n)$ where  $\xi_i$ is a pair of positive integers $(n_1(i),n_2(i))$ with the following conditions. Let $j \in [n]$ such that the $i$-th down step $s^d_i$ of $\gamma$ is one the two steps $(p_{2j-2},p_{2j-1})$ and $(p_{2j-1},p_{2j})$, and let $2k$ be the height of $p_{2j-2}$. There are three cases.
\begin{enumerate}
\item If $s^d_i = (p_{2j-2},p_{2j-1})$ such that $(p_{2j-1},p_{2j})$ is an up step (see Figure \ref{figures},(1)), then
$$k \geq n_1(i) > n_2(i) \geq 0,$$
and we weight $s^d_i$ as $\omega_i = q^{2k-n_1(i) - n_2(i)}.$
\item If $s^d_i = (p_{2j-1},p_{2j})$ such that $(p_{2j-2},p_{2j-1})$ is an up step (see Figure \ref{figures},(2)),
then $$0 \leq n_1(i) \leq n_2(i) \leq k,$$ 
and we weight $s^d_i$ as $\omega_i = q^{2k-n_1(i) - n_2(i)}.$
\item
If $(p_{2j-2},p_{2j-1})$ and $(p_{2j-1},p_{2j})$ are both down steps (see Figure \ref{figures},(3)), we can suppose that $s^d_i = (p_{2j-2},p_{2j-1})$ and $s^d_{i+1} = (p_{2j-1},p_{2j})$,
then
$$k-1 \geq n_1(i) \geq n_2(i) \geq 0,$$ 
and we weight $s^d_i$ as $\omega_i = q^{2k-1-n_1(i) - n_2(i)}$, also
$$0 \leq n_1(i+1) \leq n_2(i+1) \leq k-1,$$ 
and we weight $s^d_{i+1}$ as $\omega_{i+1} = q^{2k-2-n_1(i+1) - n_2(i+1)}$.
\end{enumerate}
\begin{figure}[!h] \center
\includegraphics[width=8cm]{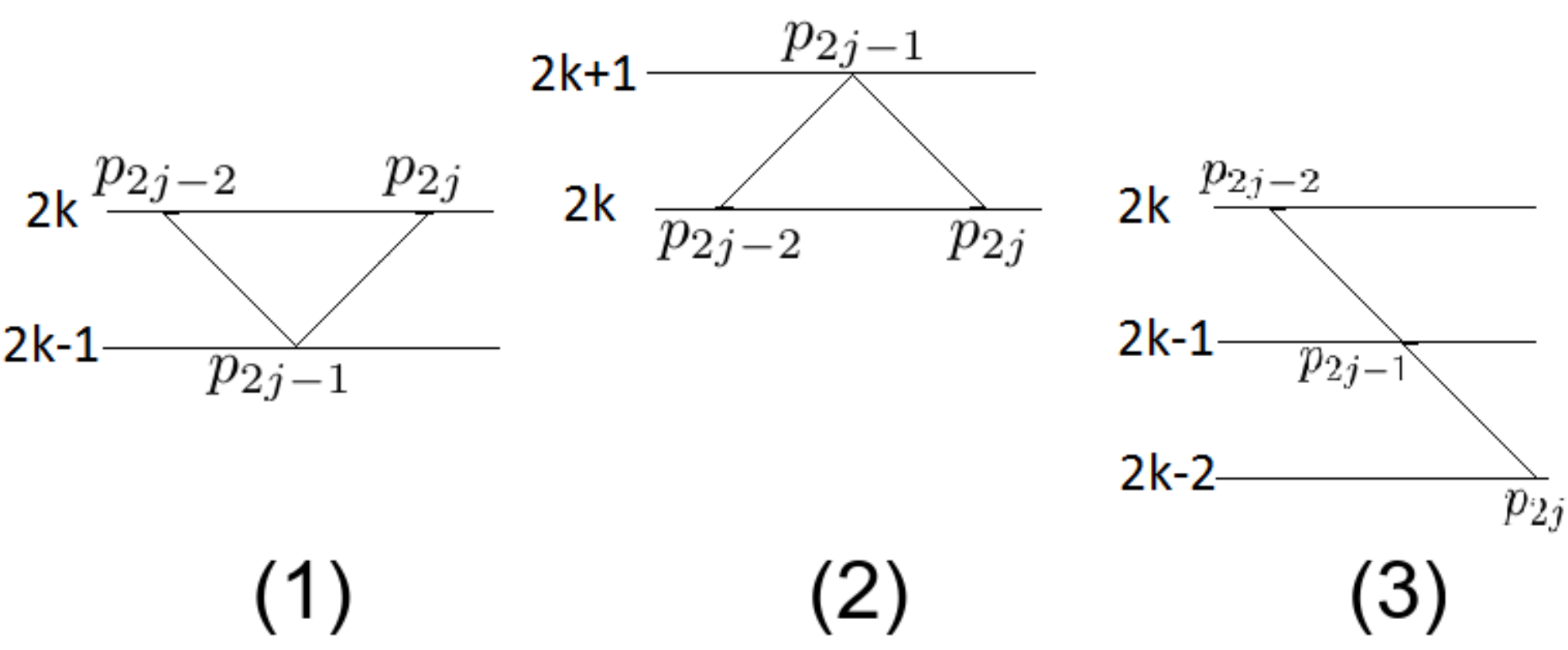}
\caption{}
 \label{figures}
\end{figure}
The \textit{weight} $\omega(\gamma, \xi)$ of the history $(\gamma, \xi)$ is the product of the weights of all down steps.
We denote by $DH(n)$ the set of Dellac histories of length $2n$.
\end{definition}

\hspace*{-5.9mm}
Prior to connecting Dellac histories to weighted Dyck paths, one can easily verify the two following results.

\begin{lemma} \label{egalitesumlambda}
For all $p \geq 1$, we have the equality
$$\sum_{0 \leq n_1 \leq n_2 \leq p-1} q^{2p-2-n_1-n_2} = (1-q^{p+1})(1-q^p)/((1-q^2)(1-q)).$$
\end{lemma}

\begin{proposition} \label{sumstoriesweighteddyckpath}
For all $\gamma_0 \in \Gamma(n)$, we have the equality
$$\sum_{(\gamma_0,\xi) \in DH(n)} \omega(\gamma_0, \xi) = \omega_{\lambda}(\gamma_0)$$
where $\omega_{\lambda}$ has been defined in (\ref{definitionomegamu}), and where $\lambda = (\lambda_n)_{n \geq 1}$ is the sequence defined in Theorem \ref{theogeneratingfunctioncn}.
\end{proposition}

\hspace*{-5.9mm}
Following Proposition \ref{sumstoriesweighteddyckpath}, we have $$\sum_{(\gamma,\xi) \in DH(n)} \omega(\gamma, \xi) = \sum_{\gamma \in \Gamma(n)} \omega_{\lambda}(\gamma)$$ for all $n \geq 0$. Therefore, from a well-known result due to Flajolet \cite{Flajolet}, the generating function $\sum_{n \geq 0} \left( \sum_{(\gamma,\xi) \in DH(n)} \omega(\gamma, \xi) \right) t^n$ is the continued fraction expansion of Formula (\ref{expansioncn+1}).
Consequently, to demonstrate Theorem \ref{theogenerfunctionhn}, it suffices to prove that $\tilde{h}_n(q) = \sum_{(\gamma,\xi) \in DH(n)} \omega(\gamma, \xi)$, which is a straight corollary of the following theorem.

\begin{theorem} \label{bijectiondellacstory}
There exists a bijective map $\Phi: DC(n) \rightarrow DH(n)$ such that
\begin{equation} \label{equationPhidellacstory}
\omega(\Phi(C)) = q^{\binom{n}{2} - \text{inv}(C)}
\end{equation}
for all $C \in DC(n)$.
\end{theorem}

\subsection{Proof of Theorem \ref{bijectiondellacstory}}
\label{sec:bijections}

In this part, we give preliminaries and connections between Dellac configurations and Dyck paths. Then, we define the algorithm $\Phi:DC(n) \rightarrow DH(n)$ and we demonstrate the statistic preservation formula (\ref{equationPhidellacstory}). Finally, we prove that $\Phi$ is bijective by giving an algorithm $\Psi : DH(n) \rightarrow DC(n)$ which happens to be $\Phi^{-1}$.

\subsubsection{\textbf{Preliminaries on Dellac configurations}}

\begin{definition}
Let $C \in DC(n)$. If $i \leq n$, we denote by $l_C^e(e_i)$ the number of inversions of $C$ between $e_i$ and any even dot $e_{i'\leq n}$ with $i' > i$.
In the same way, if $i > n$, we denote by $r_C^o(e_i)$ the number of inversions of $C$ between $e_i$ and any odd dot $e_{i'>n}$ with $i' <i$.
\end{definition}

\begin{definition}
Let $C \in DC(n)$ and $j \in [n]$. We define the \textit{height} $h(j)$ of the integer $j$ as the number $n_e(j) - n_o(j)$ where $n_e(j)$ (resp. $n_o(j)$) is the number of even dots (resp. odd dots) in the first $j-1$ columns of $C$ (with $n_e(1) = n_o(1) = 0$).
\end{definition}

\begin{remark} \label{analysedellac}
Since the first $j-1$ columns of $C$ contain exactly $2j-2$ dots and, from Remark \ref{limitationsjetons}, always contain the $j-1$ even dots $e_1,e_2, \hdots, e_{j-1}$, there exists $k \in \{0,1,\hdots,j-1\}$ such that $n_e(j) = j-1+k$ and $n_o(j) = j-1-k$. In particular $h(j) = 2k$.
\end{remark}

\begin{lemma} \label{lem2evendots}
Let $C \in DC(n)$, let $j \in [n]$ and $k \geq 0$ such that $h(j) = 2k$. If the $j$-th column of $C$ contains two odd dots, there exists $j' < j$ such $h(j'+1) = 2k$ and such that the $j'$-th column of $C$ contains two even dots.
\end{lemma}
\begin{proof}
>From Remark \ref{analysedellac}, we have $n_e(j) = j-1+k$ and $n_o(j) = j-1-k$.
Since the only $j$ odd dots that the first $j$ columns may contain are $e_{n+1},e_{n+2},\hdots,e_{n+j-1}, e_{n+j}$, and since the $j$-th column already contains two odd dots, the first $j-1$ columns contain at most $j-2$ odd dots. In other words, since they contain $n_o(j) = j - 1 - k$ odd dots, we obtain $k \geq 1$.
Thus $h(j) = 2k > 0$. Since $h(1) = 0$, there exists $j' \in [j-1]$ such that $h(j'+1) = 2k$ and $h(j') < 2k$. Obviously $h(j'+1)-h(j') \in \{-2,0,2\}$, so $h(j') = 2k-2$ and the $j'$-th column of $C$ contains two even dots.
\end{proof}

\subsubsection{\textbf{Algorithm $\Phi:DC(n) \rightarrow DH(n)$}}

\begin{definition}[$\Phi$] \label{definitionphiC}
Let $C \in DC(n)$, we define $\Phi(C)$ as $(\gamma, \xi)$, where $\gamma = (p_i)_{0 \leq i \leq 2n}$ (which is a path in $\Z^2$ whose initial point $p_0$ is defined as $(0,0)$) and $\xi = (\xi_1, \hdots, \xi_n)$ (which is a sequence of pairs of positive integers) are provided by the following algorithm. For $j=1$ to $n$, let $e_{i_1(j)}$ and $e_{i_2(j)}$ (with $i_1(j) < i_2(j)$) be the two dots of the $j$-th column of $C$.
\begin{enumerate}
\item If $i_2(j) \leq n$, then $(p_{2j-2}, p_{2j-1})$ and $(p_{2j-1},p_{2j})$ are defined as up steps.
\item
If $i_1(j) \leq n < i_2(j)$, let $i \in [n]$ such that $i-1$ down steps have already been defined. We define $\xi_{i}$ as $(l_C^e(e_{i_1(j)}), r_C^o(e_{i_2(j)})$. Afterwards,
\begin{enumerate}
\item if $l_C^e(e_{i_1(j)}) > r_C^o(e_{i_2(j)})$, we define $(p_{2j-2},p_{2j-1})$ as a down step and $(p_{2j-1}, p_{2j})$ as an up step (see Figure \ref{figures},(1));
\item if $l_C^e(e_{i_1(j)}) \leq r_C^o(e_{i_2(j)})$, we define $(p_{2j-2},p_{2j-1})$ as an up step and $(p_{2j-1}, p_{2j})$ as a down step (see Figure \ref{figures},(2)).
\end{enumerate}
\item If $n < i_1(j)$, let $i \in [n]$ such that $i-1$ down steps have already been defined. We define
$(p_{2j-2}, p_{2j-1})$ and $(p_{2j-1},p_{2j})$ as down steps (see Figure \ref{figures},(3)).
Afterwards, let $k \geq 0$ such that $p_{2j-2} = (2j-2,2k)$. Obviously, the number $n_u(2j-2) = j-1+k$ of up steps (resp. the number $n_d(2j-2) = j-1-k$ of down steps) that have already been defined is the number $n_e(j)$ of even dots (resp. the number $n_o(j)$ of odd dots) in the first $j-1$ columns of $C$, thence $h(j) = 2k$. From Lemma \ref{lem2evendots}, there exists $j' < j$ such that $h(j'+1) = 2k$ (which means $p_{2j'} = (2j',2k)$) and such that the $j'$-th column of $C$ contains two even dots, which means $(p_{2j'-2},p_{2j'-1})$ and $(p_{2j'-1},p_{2j'})$ are two consecutive up steps (see Figure \ref{samelevel}).
\begin{figure}[!h] \center
\includegraphics[width=6cm]{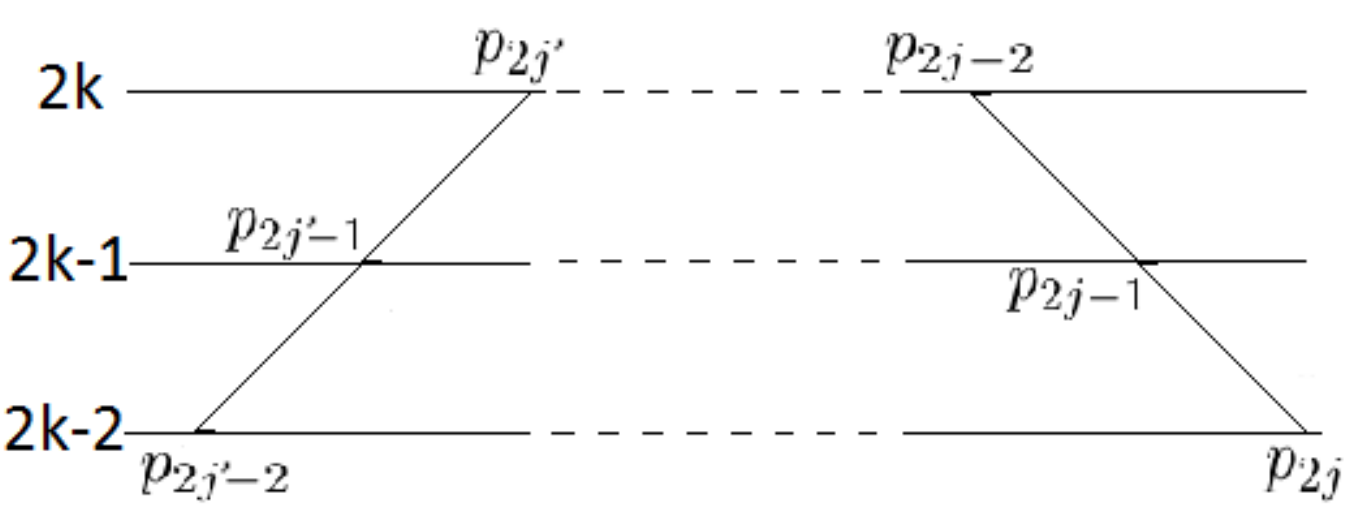}
\caption{Two consecutive up steps and down steps at the same level.}
\label{samelevel}
\end{figure}
Now, we consider the maximum $j_m < j$ of the integers $j'$ that verify this property, and we consider the two dots $e_{i_1(j_m)}$ and $e_{i_2(j_m)}$ (with $i_1(j_m) < i_2(j_m)$) of the $j_m$-th column of $C$. Finally, we define $\xi_{i}$ and $\xi_{i+1}$ as
\begin{align*}
\xi_{i}  & =  (l_C^e(e_{i_1(j_m)}),l_C^e(e_{i_2(j_m)})),\\
\xi_{i+1}  & =  (r_C^o(e_{i_1(j)}),r_C^o(e_{i_2(j)})).
\end{align*}
\end{enumerate}
\end{definition}

\begin{example} \label{exemplecalculPhiC}
The Dellac configuration $C \in DC(6)$ of Figure \ref{exempleconfigstory} yields the data $\Phi(C) = (\gamma, \xi)$, which is in fact a Dellac history, depicted in Figure \ref{imagePhiC} (since $\Phi(C)$ is a Dellac history, we have indicated the weight $\omega_i$ of the $i$-th down step $s^d_i$ of $\gamma$ for all $i \in [6]$, see Definition \ref{definitiondellacstory}).
\begin{figure}[!h]
\centering
\begin{minipage}{.4\textwidth}
  \centering
\includegraphics[width=2.5cm]{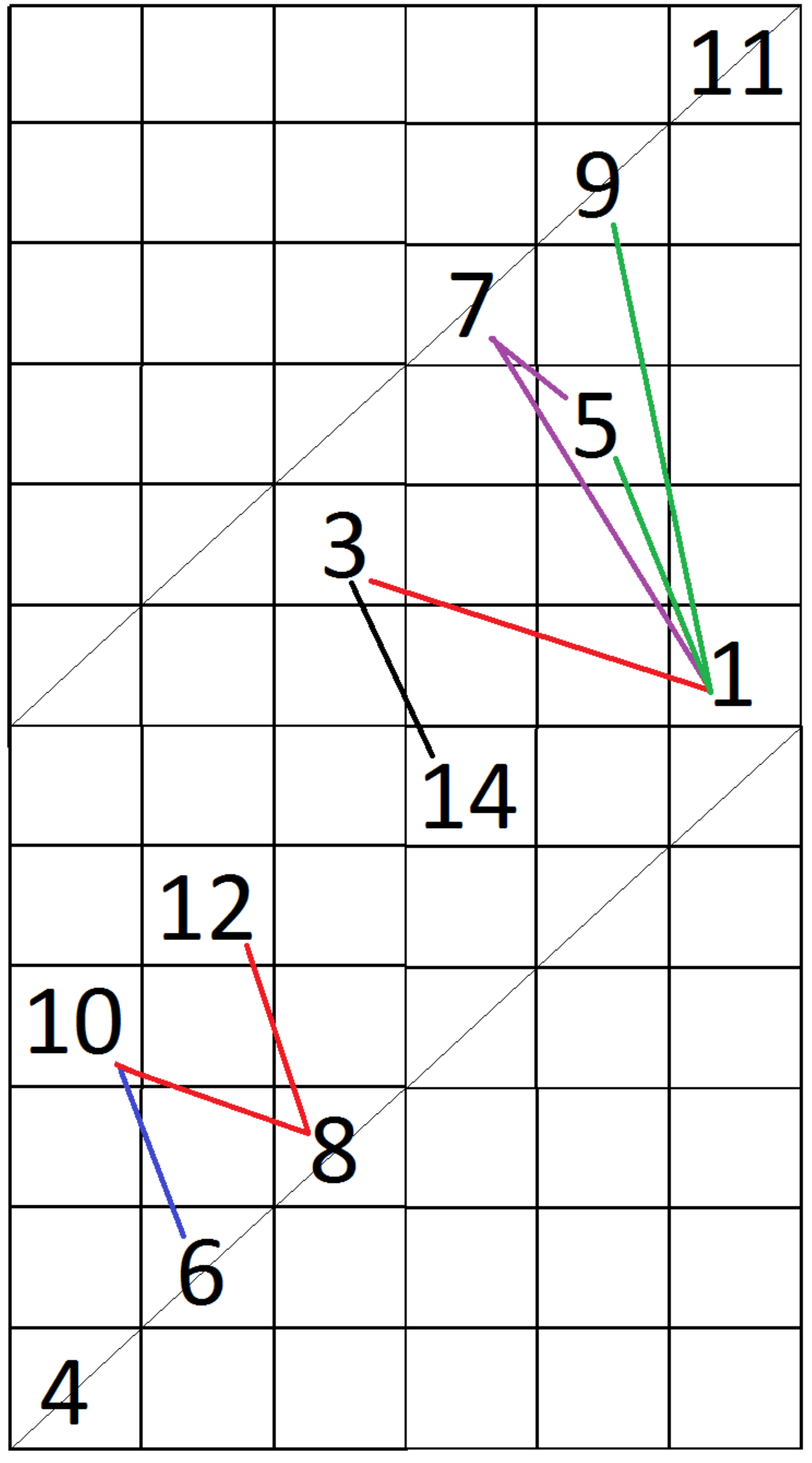}
\label{exempleconfigstory}
\caption{$C \in DC(6)$.}
\end{minipage}%
\begin{minipage}{.5\textwidth}
  \centering
\includegraphics[width=8.5cm]{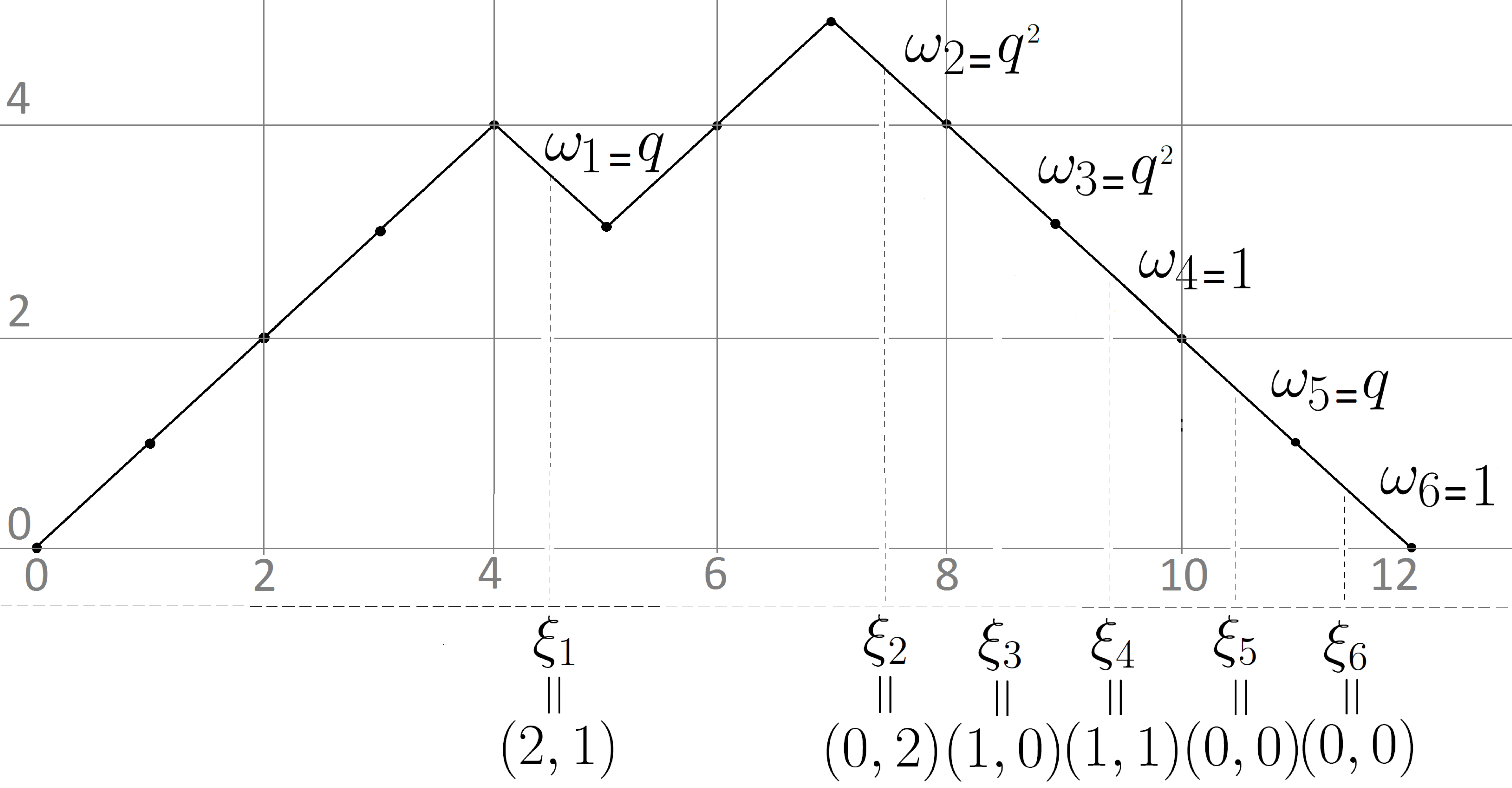}
\label{imagePhiC}
\caption{$\Psi(C) \in DH(6)$.}
\end{minipage}
\end{figure}
\end{example}

\begin{remark} \label{equivalencepasjetons}
If $\Phi(C) = (\gamma,\xi)$, there are as many up steps (resp. down steps) as even dots (resp. odd dots) in the first $j$ columns of $C$. With precision, for all $i \in [n]$, the even dot $e_{p_C(i)}$ and the odd dot $e_{n+q_C(i)}$ (see Definition \ref{defparticulardots}) give birth to the $i$-th up step and the $i$-th down step of $\gamma$ respectively.
In particular, the path $\gamma$ has $n$ up steps and $n$ down steps, so $p_{2n} = (2n,0)$. To prove that $\gamma$ is a Dyck path, we still have to check that it never goes below the line $y=0$.
\end{remark}

\begin{remark} \label{memeniveaudyckpath}
In the context $(3)$ of Definition \ref{definitionphiC}, if $h(j) = 2k$ (\textit{i.e.}, if $p_{2j-2} = (2j-2,2k)$), then the maximum $j_m$ of the integers $j' < j$ such that $h(j'+1) = 2k$ and such that the $j'$-th column contains two even dots, is such that $(p_{2j_m-2},p_{2j_m-1})$ and $(p_{2j_m-1,2j_m})$ are the last two consecutive up steps from level $2k-2$ towards level $2k$ in $\gamma$.
\end{remark}

\begin{proposition}
Let $C \in DC(n)$ and $(\gamma,\xi) = \Phi(C)$. The path $\gamma$ is a Dyck path.
\end{proposition}

\begin{proof}
From Remark \ref{equivalencepasjetons}, it suffices to prove that $\gamma = (p_0, p_1, \hdots, p_{2n})$ never goes below the line $y=0$.
If we suppose the contrary, there exists $i_0 \in \{0,1,\hdots,2n-1\}$ such that $p_{i_0} = (i_0,0)$ and $(p_{i_0},p_{i_0+1})$ is a down step. From Remark \ref{analysedyckpath}, we know that $p_{i_0} = (i_0,0) = (i_0,2n_u(i_0) - i_0)$, so $i_0 = 2 n_u(i_0)$. Let $j_0 = n_u(i_0)+1 \in [n]$.
In the first $j_0-1$ columns of $C$, from Remark \ref{equivalencepasjetons}, there are $n_u(i_0) = j_0-1$ even dots and $n_d(i_0) = j_0-1$ odd dots. Consequently, since those first $j_0-1$ columns always contain the $j_0-1$ even dots $e_1,e_2, \hdots, e_{j_0-1}$ and cannot contain any other odd dot than $e_{n+1},e_{n+2}, \hdots, e_{n+j_0-1}$ (see Remark \ref{limitationsjetons}), the $2j_0-2$ dots they contain are precisely $e_1,e_2,\hdots,e_{j_0-1}$ and $e_{n+1},e_{n+2}, \hdots, e_{n+j_0-1}$. Therefore, the only two dots that the $j_0$-th column may contain are $e_{j_0}$ and $e_{n+j_0}$. But then, it forces $l_C^e(e_{j_0})$ and $r_C^o(e_{n+j_0})$ to equal $0$. In particular $l_C^e(e_{j_0}) \leq r_C^o(e_{n+j_0})$. Following the rule $(2)(b)$ of Definition \ref{definitionphiC}, it means $(p_{i_0},p_{i_0+1})$ is defined as an up step, which is absurd by hypothesis.
\end{proof}

\begin{proposition} \label{configstory}
For all $C \in DC(n)$, the data $\Phi(C)$ is a Dellac history of length $2n$.
\end{proposition}

\begin{proof} Let $\Phi(C) = (\gamma,\xi) = ((p_0,p_1,\hdots,p_{2n}),(\xi_1,\xi_2,\hdots,\xi_n))$. We know that $\gamma \in Dyck(n)$. It remains to prove that $\xi$ fits the appropriate inequalities described in Definition \ref{definitiondellacstory}.
Let $j \in [n]$ and let $(e_{i_1(j)}$ and $e_{i_2(j)}$ (with $j \leq i_1(j) < i_2(j) \leq j+n$) be the two dots of the $j$-th column of $C$.
\begin{itemize}
\item If $(p_{2j-1},p_{2j})$ is the down step $s^d_i$ in the context $(2)(a)$ of Definition \ref{definitionphiC}, then
$\xi_{i}  = (n_1,n_2)= (l_C^e(e_{i_1(j)})), r_C^o(e_{i_2(j)}))$ with $l_C^e(e_{i_1(j)}) > r_C^o(e_{i_2(j)})$. Here, the appropriate inequality to check is $k \geq n_1 > n_2$ (this is the context $(1)$ of Definition \ref{definitiondellacstory}).
Since the first $j-1$ columns of $C$ contain $j-1+k$ even dots, including the $j-1$ dots $e_1,e_2, \hdots,e_{j-1}$ (with $j-1 < i_1(j)$), there is no inversion between any of these dots and $e_{i_1(j)}$. Consequently, in the first $j-1$ columns of $C$, there are at most $(j-1+k)-(j-1) = k$ even dots $e_i$ with $n \geq i > i_1(j)$, thence $n_1 = l_C^e(e_{i_1(j)}) \leq k$.
\item Similarly, if $(p_{2j-2},p_{2j-1})$ is the down step $s^d_i$ set in the context $(2)(b)$ of Definition \ref{definitionphiC}, then we have $\xi_{i} = (n_1,n_2) = (l_C^e(e_{i_1(j)})), r_C^o(e_{i_2(j)}))$, with $l_C^e(e_{i_1(j)}) \leq r_C^o(e_{i_2(j)})$. Now, the appropriate equality to check is $n_1 \leq n_2 \leq k$ (this is the context $(2)$ of Definition \ref{definitiondellacstory}).
The first $j$ columns of $C$ contain $j-k$ odd dots and the $i_2(j)-n$ lines from the $(n+1)$-th line to the $i_2(j)$-th line contain $i_2(j)-n$ odd dots, so, in the $n-j$ last columns, the number of odd dots $e_{i}$ with $n < i < i_2(j)$ is at most $(i_2(j) - n) - (j-k)  = k + (i_2(j) - j - n) \leq k$, thence $n_2 = r_C^o(e_{i_2(j)}) \leq k$.
\item Finally, if $(p_{2j-2},p_{2j-1})$ and $(p_{2j-1},p_{2j})$ are two consecutive down steps $s^d_i$ and $s^d_{i+1}$ in the context $(3)$ of Definition \ref{definitionphiC}, then 
\begin{align*}
\xi_{i} & = (l_C^e(e_{i_1(j_m)}),l_C^e(e_{i_2(j_m)})),\\ \xi_{i+1} & = (r_C^o(e_{i_1(j)}),r_C^o(e_{i_2(j)}))
\end{align*}
and the two inequalities to check (this is the context $(3)$ of Definition \ref{definitionphiC}) are:
\begin{align}
\label{etoile} k-1 &\geq l_C^e(e_{i_1(j_m)}) \geq l_C^e(e_{i_2(j_m)}),\\
\label{etoile2} r_C^o(e_{i_1(j)}) &\leq r_C^o(e_{i_2(j)}) \hspace*{2.3mm} \leq k-1.
\end{align} 
\begin{itemize}
\item Proof of (\ref{etoile}): since $i_1(j_m) < i_2(j_m)$, obviously $l_C^e(e_{i_1(j_m)}) \geq l_C^e(e_{i_2(j_m)})$.
Afterwards, since $p_{2j_m - 2}$ is at the level $h(j_m) = 2k-2$, there are $j_m - 1  + (k-1) = j_m+k-2$ even dots in the first $j_m-1$ columns of $C$. Since the first $j_m-1$ rows of $C$ contain the $j_m-1$ even dots $e_1,e_2,\hdots, e_{j_m-1}$, the first $j_m-1$ columns of $C$ contain at most $(j_m+k-2) - (j_m-1) = k-1$ even dots $e_i$ with $n \geq i > i_1(j_m)$, thence $l_C^e(e_{i_1(j_m)}) \leq k-1$.
\item Proof of (\ref{etoile2}): since $i_1(j) < i_2(j)$, obviously $r_C^o(e_{i_1(j)}) \leq r_C^o(e_{i_2(j)})$.
Afterwards, since $p_{2j}$ is at the level $h(j+1) = 2k-2$, there are $j  - (k-1) = j-k+1$ odd dots in the first $j$ columns of $C$. Since the $j$ rows, from the $(n+1)$-th row to the $(n+j)$-th row of $C$, contain $j$ odd dots, the $n-j$ last columns of $C$ contain at most $j - (j-k+1) = k-1$ odd dots $e_i$ with $n < i < i_2(j_m)$, thence $r_C^o(e_{i_2(j)}) \leq k-1$.
\end{itemize}
\end{itemize}
So $\Phi(C)$ is a Dellac history of length $n$.
\end{proof}

\subsubsection{\textbf{Proof of the statistic preservation formula (\ref{equationPhidellacstory})}}

Let $C \in DC(n)$ and $\Phi(C) = (\gamma,\xi)$ with $\gamma = (p_0,p_1,\hdots,p_{2n})$ and $\xi = (\xi_1,\xi_2, \hdots, \xi_{2n})$.
By definition, we have $\omega(\Phi(C)) = \Pi_{i=1}^{n} \omega_i$ where $\omega_i$ is the weight of the $i$-th down step $s^d_i$ of $\gamma$.
In the contexts $(1)$ or $(2)$ of Definition \ref{definitiondellacstory}, we have 
\begin{equation} \label{poidsdedi}
\omega_i = q^{2k-l_C^e\left(e_{i_1(j)}\right) - r_C^o\left(e_{i_2(j)}\right)}.
\end{equation}
Since $p_{2j-2}$ is at the level $h(j) = 2k$, the first $j-1$ columns of $C$ contain $j-1-k$ odd dots.
Consequently, following Definition \ref{definitionphiC}, the step $s^d_i$ is the $(j-k)$-th down step of $\gamma$, \textit{i.e.}, the integer $i$ equals $j-k$.
Also, since the first $j$ columns of $C$ contain $j+k$ even dots, the last $n-j$ columns of $C$ (from the $(j+1)$-th column to the $n$-th column) contain $n - (j+k) = n-j-k = i -k$ even dots.
As a result, we obtain the equality
\begin{equation} \label{bCenfonctiondebCodd}
r_C(e_{i_2(j)}) = r_C^o(e_{i_2(j)}) + i-k.
\end{equation}
In view of (\ref{bCenfonctiondebCodd}), Equality (\ref{poidsdedi}) becomes $\omega_i = q^{n-i - \left(l_C^e \left( e_{i_1(j)} \right) + r_C \left( e_{i_2(j)} \right) \right)}$. 
With the same reasoning, if $s^d_i$ and $s^d_{i+1}$ are two consecutive down steps in the context $(3)$ of Definition \ref{definitiondellacstory}, then by commuting factors of $\omega_i$ and $\omega_{i+1}$, we obtain the equality
$$\omega_i \omega_{i+1} = \left( q^{n-i - \left(l_C^e \left( e_{i_1(j_m)} \right) + r_C \left( e_{i_2(j_m)} \right) \right)} \right)  
\left(q^{n-(i+1) - \left(l_C^e \left( e_{i_1(j)} \right) + r_C \left( e_{i_2(j)} \right) \right)} \right).$$
>From $\omega(\Phi(C) = \Pi_{i = 1}^n \omega_i$, it follows that
\begin{equation} \label{poidsphiCpresque}
\omega(\Phi(C)) = q^{ \left( \sum_{i=1}^n n-i \right)  - \left( \sum_{i \leq n} l_C^e(e_i) + \sum_{i > n} r_C(e_i) \right)}.
\end{equation}
Now, it is easy to see that
$\text{inv}(C) = \sum_{i \leq n} l_C^e(e_i) + \sum_{i > n} r_C(e_i)$. In view of the latter remark, Formula (\ref{poidsphiCpresque}) becomes Formula (\ref{equationPhidellacstory}). \hfill $\qed$

\subsubsection{\textbf{Proof of the bijectivity of $\Phi:DC(n) \rightarrow DH(n)$}}

To end the proof of Theorem \ref{bijectiondellacstory}, it remains to show that $\Phi$ is bijective. To this end, we construct (in Definition \ref{defPsiS}) a map $\Psi : DH(n) \rightarrow DC(n)$ and we prove in Lemma \ref{PhiandPsiinversemaps} that $\Phi$ and $\Psi$ are inverse maps.

\begin{definition} \label{defPsiS}
Let $S = (\gamma,\xi) \in DH(n)$ with $\gamma = (p_0,p_1,\hdots, p_{2n})$ and $\xi = (\xi_1, \hdots, \xi_{n})$.
We define $\Psi(S)$ as a tableau $T$ of width $n$ and height $2n$, in which we insert the $2n$ dots $e_1,e_2, \hdots, e_{2n}$ according to the two following (analogous and independant) algorithms.
\begin{enumerate}
\item \textbf{Insertion of the $n$ odd dots $e_{n+1},e_{n+2}, \hdots, e_{2n}$.} Let $\I_0^o = (1, 2, \hdots, n)$.
For $i = 1$ to $n$, consider $j_i \in [n]$ such that the $i$-th down step $s^d_i$ of $\gamma$ is one of the two steps $(p_{2j_i-2},p_{2j_i-1})$ or $(p_{2j_i-1},p_{2j_i})$. If the set $\I_{i-1}^o \subset \I_0^o$ is defined, we denote by $H(i)$ the hypothesis "$\I_{i-1}^o$ has size $n+1-i$ such that for all $j \in \{i,i+1, \hdots, n\}$, the $(j-i+1)$-th element of $\I_{i-1}^o$ is inferior to $n+j$". If the hypothesis $H(i+1)$ is true, then we iterate the algorithm to $i+1$.
At the beginning, $\I_0^o$ is defined and $H(1)$ is obviously true so we can initiate the algorithm.
\begin{enumerate}
\item If $s^d_i$ is a down step in the context $(1)$ or $(2)$ of Definition \ref{definitiondellacstory}, let $(n_1,n_2) = \xi_i$. In particular, since $n_2 \leq k = j_i-i$ (see Remark \ref{relationijk}) and $j_i \leq n$, we have $1+n_2 \leq n-i+1$ so, from Hypothesis $H(i)$, we can consider the $(1+n_2)$-th element of $\I_{i-1}^o$, say, the integer $q$.
We insert the odd dot $e_{n+q}$ in the $j_i$-th column of $T$. From Hypothesis $H(i)$, the $(j_i-i+1)$-th element of $\I_{i-1}^o$ is inferior to $n+j_i$, and $1+n_2 \leq 1+k = j_i-i+1$. Consequently, the dot $e_{n+q}$ is between the lines $y = x$ and $y = x+n$.
Afterwards, we define $\I_i^o$ as the sequence $\I_{i-1}^o$ from which we have removed $q$ (by abusing the notation, we write $\I_i^o := \I_{i-1}^o \backslash \{q\}$).
Thus, the set $\I_i^o$ has size $n+1-(i+1)$. Also, if $j \in \{i+1, i+2, \hdots, n\}$, then following Hypothesis $H(i)$, the $(j-i)$-th element of $\I_{i-1}^o$ is inferior to $n+j-1$, so the $(j-(i+1)+1)$-th element of $\I_i^o$ is inferior to $n+j-1 < n+j$. Therefore, Hypothesis $H(i+1)$ is true and we can iterate the algorithm to $i+1$.
\item
If $s^d_i$ and $s^d_{i+1}$ are two consecutive down steps in the context $(3)$ of Definition \ref{definitiondellacstory}, let $(n_1,n_2) = \xi_{i+1}$. In particular $n_1 \leq n_2 \leq k-1 = j_i-i-1 \leq n-i-1$, so $1+n_1 < 2+n_2 \leq j_i-i+1$. Consequently, following Hypothesis $H(i)$, we can consider the $(1+n_1)$-th element of $\I_{i-1}^o$, say, the integer $q_1$, and the $(2+n_2)$-th element of $\I_{i-1}^o$, say, the integer $q_2 > q_1$.
We insert the two odd dots $e_{n+q_1}$ and $e_{n+q_2}$ in the $j$-th column of $T$. With precision, by the same argument as for $(a)$, those two dots are located between the lines $y=x$ and $y=x+n$.
Afterwards, we set $\I_{i+1}^o := \I_{i-1}^o \backslash \{q_1,q_2\}$.
Thus, the $\I_{i+1}^o$ has size $n - (i+2) +1$, and if $j \in \{i+2, i+3, \hdots, n\}$ then, by Hypothesis $H(i)$, the $(j-i-1)$-th element of $\I_{i-1}^o$ is inferior to $n+j-2$, so the $(j-(i+2)+1)$-th element of $\I_{i+1}^o$ is inferior to $n+j-2 < n+j$. Therefore, Hypothesis $H(i+2)$ is true and we can iterate the algorithm to $i+2$.
\end{enumerate}
\item \textbf{Insertion of the $n$ even dots $e_1,e_2, \hdots,e_n$.} Let $\I_0^e = (n, n-1, \hdots, 1)$. For $i=1$ to $n$, consider $j_i \in [n]$ such that the $(n+1-i)$-th up step $s^u_{n+1-i}$ of $\gamma$ is one of the two steps $(p_{2j_i-2},p_{2j_i-1})$ or $(p_{2j_i-1},p_{2j_i})$. If the set $\I_{i-1}^e \subset \I_0^e$ is defined, we denote by $H'(i)$ the hypothesis "$\I_{i-1}^e$ has size $n+1-i$ such that for all $j \in [n-i+1]$, the $(n-i+2-j)$-th element of $\I_{i-1}^o$ is greater than $j$". If Hypothesis $H'(i+1)$ is true, we iterate the algorithm to $i+1$.
In particular, the set $\I_0^e$ is defined and $H'(1)$ is true so we can initiate the algorithm.
\begin{enumerate}
\item If $s^u_{n+1-i}$ is an up step in the the context $(1)$ or $(2)$ of Definition \ref{definitiondellacstory}, then let $i_0 \in [n]$ such that
 $\{ (p_{2j_i-2},p_{2j_i-1}), (p_{2j_i-1},p_{2j_i}) \} = \{s^u_{n+1-i}, s^d_{i_0} \}$.
Let $(n_1,n_2) = \xi_{i_0}$. From Remark \ref{relationijk}, we have $1+n_1 \leq 1+k = n-i+2-j_i \leq n-i+1$ so, following Hypothesis $H'(i)$, we can consider the $(1+n_1)$-th element of $\I_{i-1}^e$, say, the integer $p$.
We insert the even dot $e_{p}$ in the $j_i$-th column of $T$. By Hypothesis $H'(i)$, the $(n-i+2-j_i)$-th element of $\I_{i-1}^e$ is greater than $j_i$, and $1+n_1 \leq 1+k = n-i-j_i+2$ so the dot $e_{p}$ is located between the lines $y = x$ and $y = x+n$.
Afterwards, we set $\I_i^e := \I_{i-1}^e \backslash \{p\}$.
The set $\I_i^e$ has size $n+1-(i+1)$. Also, if $j \in \{1, 2, \hdots, n+1-(i+1)\}$, then, by Hypothesis $H'(i)$, the $(n-i-j)$-th element of $\I_{i-1}^e$ is greater than $j+1$, so the $(n-(i+1)+1-j)$-th element of $\I_i^e$ is greater than $j+1 > j$.
Therefore, Hypothesis $H'(i+1)$ is true and we can iterate the algorithm to $i+1$.
\item If $s^u_{n+1-(i+1)}$ and $s^u_{n+1-i}$ are two consecutive up steps $(p_{2j_i-2},p_{2j_i-1})$ and $(p_{2j_i-1},p_{2j_i})$ from level $2k-2$ towards level $2k$ in $\gamma$, let $j_0 > j_i$ such that the two steps $(p_{2j_0-2},p_{2j_0-1})$ and $(p_{2j_0-1},p_{2j_0})$ are the next two consecutive down steps $s^d_{i_0}$ and $s^d_{i_0+1}$ from level $2k$ towards level $2k-2$ (see Figure \ref{samelevel}).
Let $(n_1,n_2) = \xi_{i_0}$. Being in the context $(3)$ of Definition \ref{definitiondellacstory}, we have $n_2 \leq n_1 \leq k-1 = n-i-j_0 \leq n-i-1$, hence $1+n_2 < 2+n_1 \leq n-i+1$. Consequently, by Hypothesis $H'(i)$, we can consider the $(1+n_2)$-th element of $\I_{i-1}^e$, say, the integer $p_1$, and the $(2+n_1)$-th element of $\I_{i-1}^e$, say, the integer $p_2 < p_1$. 
We insert the two even dots $e_{p_2}$ and $e_{p_1}$ in the $j_i$-th column of $T$. With precision, for the same argument as for $(a)$, those two dots are between the lines $y=x$ and $y=x+n$.
Afterwards, we set $\I_{i+1}^e := \I_{i-1}^e \backslash \{p_2,p_1\}$.
The set $\I_{i+1}^e$ has size $n - (i+2) +1$. Also, if $j \in \{1, 2, \hdots, n+1-(i+2)\}$, then by Hypothesis $H'(i)$, the $(n-i-j)$-th element of $\I_{i-1}^e$ is greater than $j+2$, so the $(n-(i+2)+2-j)$-th element of $\I_{i+1}^e$ is greater than $j+2 >j$.
Therefore, Hypothesis $H'(i+2)$ is true and we can iterate the algorithm to $i+2$.
\end{enumerate}
\end{enumerate}
By construction, it is clear that $\Psi(S) = T$ is a Dellac configuration.
\end{definition}

\begin{remark} \label{correspondancejetonpas}
Let $S = (\gamma, \xi) \in DH(n)$ and $C = \Psi(S) \in DC(n)$. For all $i \in [n]$, the $i$-th up step $s^u_i$ (resp. down step $s^d_i$) of $\gamma$ gives birth to the even dot $e_{p_C(i)}$ (resp. to the odd dot $e_{n+q_C(i)}$) (see Definition \ref{defparticulardots}).
\end{remark}

\begin{example}
If $S \in DH(6)$ is the Dellac history $\Phi(C)$ of Example \ref{exemplecalculPhiC}, we obtain $\Psi(S) = C$. 
\end{example}

\hspace*{-5.9mm} 
Following Remark \ref{correspondancejetonpas}, it is easy to prove the following lemma by induction on $i \in [n]$.

\begin{lemma} \label{lemaCevenbCoddedCod}
Let $S \in DH(n)$. We consider the two sequences $(\I_i^o)$ and $(\I_i^e)$ defined in the computation of $C = \Psi(S)$ (see Definition \ref{defPsiS}). Then for all $i \in [n]$,
the integer $q_C(i)$ is the $(1+r_C^o(e_{n+q_C(i)}))$-th element of the sequence $\I_{i-1}^o$, and the integer $p_C(n+1-i)$ is the $(1+l_C^e(e_{p_C(n+1-i)}))$-th element of the sequence $\I_{i-1}^e$.
\end{lemma}

\begin{proposition} \label{PhiandPsiinversemaps}
The maps $\Phi : DC(n) \rightarrow DH(n)$ and $\Psi : DH(n) \rightarrow DC(n)$ are inverse maps.
\end{proposition}

\begin{proof} From Remarks \ref{equivalencepasjetons} and \ref{correspondancejetonpas},
it is easy to see that $\Phi \circ \Psi = Id_{DH(n)}$. The equality $\Psi \circ \Phi = Id_{DC(n)}$ is less straightforward.
Let $C \in DC(n)$ and $S = (\gamma,\xi) = \Phi(C) \in DH(n)$.
We are going to show, by induction on $i \in [n]$, that $q_{\Psi(S)}(i) = q_C(i)$ and $p_{\Psi(S)}(i) = p_C(i)$ for all $i$, hence $\Psi(S) = C$.
The two proofs of $q_{\Psi(S)}(i) = q_C(i)$ and $p_{\Psi(S)}(i) = p_C(i)$ respectively being independant and analogous, we only prove $q_{\Psi(S)}(i) = q_C(i)$ for all $i$. Let $i=1$. In the context $(1)(a)$ of Definition \ref{defPsiS}, from Remark \ref{equivalencepasjetons}, the first odd dot to be inserted is $e_{n+q_{\Psi(S)}(1)}$. Therefore, by definition, the integer $q_{\Psi(S)}(1)$ is the $(1+n_2)$-th element of $\I_0^o$ (\textit{i.e.}, we obtain $q_{\Psi(S)}(1)$ = $1+n_2$ where $(n_1,n_2) = \xi_1$. In this situation, since $S = \Phi(C)$, we know that $n_2 = r_C^o(e_{n+q_C(1)})$.
Consequently, from Lemma \ref{lemaCevenbCoddedCod}, we obtain $q_{\Psi(S)}(1) = 1+ r_C^o(e_{n+q_C(1)}) = q_C(1).$
The proof in the context $(1)(b)$ is analogous. Now let $i \in \{2, 3, \hdots, n\}$. Suppose that $q_{\Psi(S)}(k) = q_C(k)$ for all $k < i$.
In the context $(1)(a)$ of Definition \ref{defPsiS}, from Remark \ref{equivalencepasjetons}, the $i$-th odd dot to be inserted is $e_{n+q_{\Psi(S)}(i)}$. Therefore, by definition, if $\xi_i = (n_1,n_2)$, then $q_{\Psi(S)}$ is the $(1+n_2)$-th element of $\I_{i-1}^e = \J_{i-1}^e$. Since $S = \Phi(C)$, we know that $n_2 = r_C^o(e_{n+q_C(i)})$ so, from Lemma \ref{lemaCevenbCoddedCod}, we obtain $q_{\Psi(S)}(i) = q_C(i)$.
The proof in the context $(1)(b)$ is analogous.
\end{proof}
\\ \ \\
This puts an end to the proof of Theorem \ref{bijectiondellacstory}.
As an illustration of the entire paper, the table depicted in the next page (see Figure \ref{table}) explicits the bijections $\phi : DC(3) \rightarrow \D_4'$ and $\Phi : DC(3) \rightarrow DH(3)$.

\subsection*{Acknowledgements}
I thank Jiang Zeng for his comments and useful references.

\clearpage

\begin{figure}
\begin{center}
\begin{tabular}{c|c|c}
    \toprule
  $C \in DC(3)$ & 
   $\phi(C) \in \D_4'$ & 
    $\Phi(C) \in DH(3)$ \\
    \midrule
    \adjustimage{height=2.5cm,valign=m}{C0} &
    $41736285$ &
    \adjustimage{height=2.5cm,valign=m}{S0} \\
    \midrule
   \adjustimage{height=2.5cm,valign=m}{C11} &
   $41736582$ &
    \adjustimage{height=2.5cm,valign=m}{S11} \\
     \midrule
   \adjustimage{height=2.5cm,valign=m}{C12} &
   $71436285$ &
    \adjustimage{height=2.5cm,valign=m}{S12} \\
     \midrule
   \adjustimage{height=2.5cm,valign=m}{C21} &
   $71436582$ &
    \adjustimage{height=2.5cm,valign=m}{S21} \\
     \midrule
   \adjustimage{height=2.5cm,valign=m}{C22} &
   $51436287$ &
    \adjustimage{height=2.5cm,valign=m}{S22} \\
     \midrule
   \adjustimage{height=2.5cm,valign=m}{C23} &
   $21736584$ &
    \adjustimage{height=2.5cm,valign=m}{S23} \\
     \midrule
   \adjustimage{height=2.5cm,valign=m}{C3} &
   $21436587$ &
    \adjustimage{height=2.5cm,valign=m}{S3} \\
    \bottomrule
\end{tabular}
\end{center}
\caption{}
\label{table}
\end{figure}

\clearpage



\end{document}